\documentclass{svjour3}                     
\smartqed  
\usepackage{graphicx}
\usepackage{mathptmx}      
\renewcommand{\subclassname}{%
  \textup{2010} Mathematics Subject Classification\quad}

\usepackage{amsmath}
\usepackage{amscd}
\usepackage{amssymb}
\usepackage{cases} 

\newsavebox{\SmallMathBox}

\DeclareRobustCommand*{\nicefrac}[2]{\ifmmode\mathnicefrac{#1}
{ #2}%
  \else\textnicefrac{#1}{#2}\fi}
\newcommand*{\textnicefrac}[2]{\check@mathfonts%
\mbox{\raisebox{.5ex}{\fontsize\sf@size\z@\selectfont#1}\kern-.
1em%
/\kern-.1em\raisebox{- .25ex}{\fontsize\sf@size\z@\selectfont#2} }}
\newcommand*{\mathnicefrac}[2]{%
  \mathchoice
    {\m@fr@c{\scriptstyle}{#1}{#2}}
    {\m@fr@c{\scriptstyle}{#1}{#2}}
    {\m@fr@c{\scriptscriptstyle}{#1}{#2}}
    {\m@fr@c{\scriptscriptstyle}{#1}{#2}}}

\newcommand{\abs}[1]{\lvert#1\rvert}

\def\Ci{C^\infty}

\def\fequal#1{\stackrel{#1}{=}}

\def\into{\hookrightarrow}
\def\lla{\langle}

\def\noi{\noindent}
\newcommand{\norm}[1]{\lVert#1\rVert}

\def\ol{\overline}

\def\rra{\rangle}
\def\sqm1{\sqrt{-1}}

\def\tand{\mbox{\ \rm  and }}

\def\too{\longrightarrow}

\def\wt{\widetilde}

\def\={\cong}
\def\>{\supset}
\def\<{\subset}
\def\ii{^{-1}}
\def\12{\frac{1}{2}}
\def\0{^{\circ}}

\def\C{\CC}
\renewcommand*{\d}{\delta}
\newcommand*{\D}{\Delta}
\def\e{\varepsilon}
\def\f{\varphi}

\def\g{\gamma}

\def\la{\lambda}

\def\m{\mu}

\def\N{\NN}

\def\R{\RR}
\def\s{\sigma}
\def\Si{\Sigma}

\def\w{\omega}

\def\z{\zeta}
\def\Z{\ZZ}

\def\CC{{\mathbb C}}

\def\NN{{\mathbb N}}

\def\RR{{\mathbb R}}

\def\ZZ{{\mathbb Z}}

\def\Aa{{\mathcal A}}
\def\Bb{{\mathcal B}}
\def\Cc{{\mathcal C}}

\def\Ff{{\mathcal F}}

\def\Ll{{\mathcal L}}

\def\Pp{{\mathcal P}}

\def\Ss{{\mathcal S}}

\def\Uu{{\mathcal U}}

\def\CLR{{\mathcal CLR}}

\def\GGG{{\mathfrak G}}

\newcommand\Calderon{Cal\-der{\'o}n}

 \DeclareMathOperator{\dist}{dist}
\DeclareMathOperator{\diag}{diag}
\DeclareMathOperator{\dom}{dom}

 \DeclareMathOperator{\image}{im}
\def\index{\mbox{\rm index\,}}

\DeclareMathOperator{\Mas}{Mas} \DeclareMathOperator{\mmax}{max}

\DeclareMathOperator{\ran}{im}
\DeclareMathOperator{\range}{im}

\DeclareMathOperator{\SF}{sf}

\journalname{}
\begin{document}

\setcounter{page}{1} \setcounter{tocdepth}{2}

\renewcommand*{\labelenumi}{%
   (\roman{enumi})}


\title{The Maslov index in weak symplectic functional analysis
\thanks{The second author was partially supported by KPCME No. 106047 and NNSF
No. 10621101.}}

\titlerunning{Weak Symplectic Functional Analysis}        

\author{Bernhelm Boo{\ss}-Bavnbek
\and Chaofeng Zhu}
\authorrunning{B. Boo{\ss}-Bavnbek and C. Zhu} 

\institute{B. Boo{\ss}-Bavnbek \at Department of Science, Systems
and Models/IMFUFA\\ Roskilde University, DK-4000 Ros\-kilde, Denmark\\
\email{booss@ruc.dk} \and
C. Zhu \at Chern Institute of Mathematics and LPMC\\
Nankai University, Tianjin 300071, the People's Republic of China\\ \email{zhucf@nankai.edu.cn} }


\maketitle

\begin{abstract}
We recall the Chernoff-Marsden definition of weak symplectic structure and
give a rigorous treatment of the functional analysis
and geometry of weak symplectic Banach spaces. We define
the Maslov index of a continuous path of Fredholm pairs of Lagrangian
subspaces in continuously varying Banach spaces.
We derive basic properties of this Maslov index and emphasize the
new features appearing.
%
%
%
%
\end{abstract}

\maketitle




\section{Introduction}


\subsection{Our setting and goals}
First, we recall the main features of finite-dimensional and
infinite-dimensional strong symplectic analysis and geometry and argue for the
need to generalize from strong to weak assumptions.

\subsubsection{The finite-dimensional case}
The study of dynamical systems and the variational calculus of $N$-particle classical mechanics
automatically lead to a symplectic structure in the phase space $X=\RR^{6N}$ of position and
impulse variables: when we trace the motion of $N$ particles in $3$-dimensional space, we deal with
a bilinear (in the complex case sesquilinear) anti-symmetric (in the complex case skew-symmetric)
and non-degenerate form $\w\colon X\times X\to\RR$. The reason for the skew-symmetry is the
asymmetry between position and impulse variables corresponding to the asymmetry of differentiation.
To carry out the often quite delicate calculations of mechanics, the usual trick is to replace the
skew-symmetric form $\w$ by a skew-symmetric matrix $J$ with $J^2=-I$ such that
\begin{equation}\label{e:omega2j}
\w(x,y)\ =\ \lla Jx,y\rra\quad\text{ for all $x,y\in X$},
\end{equation}
where $\lla\cdot,\cdot\rra$
denotes the inner product in $X$.

For geometric investigations, the key concept is a Lagrangian subspace of the phase space. For two
continuous paths of Lagrangian subspaces, an intersection index, the {\em Maslov index} is
well-defined. It can be considered as a re-formulation or generalization of counting conjugate
points on a geodesic. In Morse Theory, this number equals the classical Morse index, i.e., the
number of negative eigenvalues of the Hessian (the second variation of the action/energy
functional). This Morse Index Theorem (cf. M. Morse \cite{Mo}) for geodesics on Riemannian
manifolds was extended by W. Ambrose \cite{Am}, J.J. Duistermaat \cite{Du76}, P. Piccione and D.V.
Tausk \cite{PiT1,PiT2}, and the second author \cite{Zh01,Zh05}. See also the work of M. Musso, J.
Pejsachowicz, and A. Portaluri on a Morse index theorem for perturbed geodesics on semi-Riemannian
manifolds in \cite{MuPePo} which has in particular lead N. Waterstraat to a $K$-theoretic proof of
the Morse Index Theorem in \cite{Wa}.

For a systematic review of the basic vector analysis and geometry and for the physics background,
we refer to V.I. Arnold \cite{Ar78} and M. de Gosson \cite{Go01}.

\subsubsection{The strong symplectic infinite-dimensional case}
As shown by K. Furutani and the first author in \cite{BoFu98}, the finite-dimensional approach of
the Morse Index Theorem can be generalized to a separable Hilbert space when we assume that the
form $\w$ is bounded and can be expressed as in \eqref{e:omega2j} with a bounded operator $J$,
which is skew-self-adjoint (i.e., $J^*=-J$) and not only injective but invertible. The
invertibility of $J$ is the whole point of a {\em strong} symplectic structure. Then, without loss
of generality, one can assume $J^2=-I$ like in the finite-dimensional case (see Lemma
\ref{l:weak-strong} below), and many calculations of the finite-dimensional case can be preserved
with only slight modifications. The model space for strong symplectic Hilbert spaces is the von
Neumann space $\beta(A):=\dom(A^*)/\dom(A)$ of {\em natural} boundary values of a closed symmetric
operator $A$ in a Hilbert space $X$ with symplectic form given by Green's form
\begin{equation}\label{e:omega-beta}
\w(\g(u),\g(v)):\ =\  \lla A^*u,v\rra - \lla u,A^*v\rra  \quad\text{
for all $u,v\in \dom(A^*)$},
\end{equation}
where $\lla\cdot,\cdot\rra$ denotes the inner product in $X$ and $\g\colon \dom(A^*)\to\beta(A)$ is
the trace map. A typical example is provided by a linear symmetric differential operator $A$ of
first order over a manifold $M$ with boundary $\Si$. Here we have the minimal domain
$\dom(A)=H_0^1(M)$ and the maximal domain $\dom(A^*)\>H^1(M)$. Note that the inclusion is strict
for $\dim M>1$. Recall that $H_0^1(M)$ denotes the closure of $\Ci_0(M\setminus \Si)$ in $H^1(M)$.
For better reading we do not mention the corresponding vector bundles in the notation of the
Sobolev spaces of vector bundle sections.

As in the finite-dimensional case, the basic geometric concept in infinite-dimensional strong
symplectic analysis is the  Lagrangian subspace, i.e., a subspace which is isotropic and
co-isotropic at the same time. Contrary to the finite-dimensional case, however, the common
definition of a Lagrangian as a {\em maximal} isotropic space or an isotropic space of {\em half}
dimension becomes inappropriate.

In order to define the Maslov index in the infinite-dimensional case as intersection number of two
continuous paths of Lagrangian subspaces, one has to make the additional assumption that
corresponding Lagrangians make a Fredholm pair so that, in particular, we have finite intersection
dimensions.

In \cite{Fl88}, A. Floer suggested to express the spectral flow of a curve of self-adjoint
operators by the Maslov index of corresponding curves of Lagrangians. Following his suggestion, a
multitude of formulae was achieved by T. Yoshida \cite{Yo91}, L. Nicolaescu \cite{Ni95}, S. E.
Cappell, R. Lee, and E. Y. Miller \cite{CaLeMi96}, the first author, jointly with K. Furutani and
N. Otsuki \cite{BoFu99,BoFuOt01} and P. Kirk and M. Lesch \cite{KiLe00}. The formulae are  of
varying generality: Some deal with a fixed (elliptic) differential operator with varying
self-adjoint extensions (i.e., varying boundary conditions); others keep the boundary condition
fixed and let the operator vary. An example for a path of operators is a curve of Dirac operators
on a manifold with fixed Riemannian metric and Clifford multiplication but varying defining
connection (background field). See also the results by the present authors in \cite{BoZh05} for
varying operator and varying boundary conditions but fixed maximal domain and in \cite{BoZh10b} (in
preparation) also for varying maximal domain. Recently, M. Prokhorova \cite{Pro:SFD} considered a
path of Dirac operators on a two-dimensional disk with a finite number of holes subjected to local
elliptic boundary conditions and obtained a beautiful explicit formula for the spectral flow
(respectively, the Maslov index).

\subsubsection{Beyond the limits of the strong symplectic assumption}
Weak (i.e., not necessarily strong) symplectic structures arise on the way to a spectral flow
formula in the full generality wanted: for continuous curves of, say linear formally self-adjoint
elliptic differential operators of first order over a compact manifold of dimension $\ge 2$ with
boundary and with varying maximal domain (i.e., admitting arbitrary continuous variation of the
coefficients of first order) and with continuously varying regular (elliptic) boundary conditions,
see \cite{BoZh10b}.
An interesting new feature for the comprehensive generalization is the following ``technical"
problem: For regular (elliptic) boundary value problems (say for a linear formally self-adjoint
elliptic differential operator $A$ of first order on a compact smooth manifold $M$ with boundary
$\Si$), there are three canonical spaces of boundary values: the above mentioned von Neumann space
$\beta(A)=\dom(A^*)/\dom(A)$, which is a subspace of the distributional Sobolev space
$H^{-1/2}(\Si)$; the space of boundary values $H^{1/2}(\Si) \simeq H^1(M)/H^1_0(M)$ of the operator
domain $H^1(M)$; and the most familiar and basic $L^2(\Si)$.\footnote{In the tradition of
geometrically inspired analysis, we think mostly of {\em homogeneous} systems when talking of
elliptic boundary value problems. Our key reference is the monograph \cite{BoWo93} by K. P.
Wojciechowski and the first author and the supplementary elaborations by J. Br{\"u}ning and M.
Lesch in \cite{BrLe01}. For a more comprehensive treatment, emphasizing {\em non-homogeneous}
boundary value problems and assembling all relevant section spaces in a huge algebra, we refer to
the more recent article \cite{Sch01} by B.-W. Schulze.} As in \eqref{e:omega-beta}, Green's form
induces symplectic forms on all three section spaces which are mutually compatible.

More precisely, Green's form yields a strong
symplectic structure not only on $\beta(A)$, but also on $L^2(\Si)$ by
\[
\w(x,y) :\ =\  -\lla Jx,y\rra_{L^2(\Si)}\,.
\]
Here $J$ denotes the principal symbol of the operator $A$ over the boundary in inner normal
direction. The multiplicative operator induced by $J$ is invertible (= injective and surjective,
i.e., with bounded inverse) since $A$ is elliptic. For the induced symplectic structure on the
Sobolev space $H^{1/2}(\Si)$ the corresponding operator $J'$ is {\em not} invertible for
$\dim\Si\ge 1$, see Remark \ref{r:weak-strong}b in Section \ref{ss:basic-symplectic} below. So, for
$\dim \Si\ge 1$ the space $H^{1/2}(\Si)$ becomes only a {\em weak} symplectic Hilbert space, to use
a notion introduced by P.R. Chernoff and J.E. Marsden \cite[Section 1.2, pp. 4-5]{ChMa74}.

An additional incitement to investigate weak symplectic structures comes from a stunning
observation of E. Witten (explained by M.F. Atiyah in \cite{At85} in a heuristic way). He
considered a weak (and degenerate) symplectic form on the loop space $\operatorname{Map}(S^1,M)$ of
a finite-dimensional closed orientable Riemannian manifold $M$ and noticed that a (future) thorough
understanding of the infinite-dimensional symplectic geometry of that loop space ``should lead
rather directly to the index theorem for Dirac operators" (l.c., p. 43). Of course, restricting
ourselves to the linear case, i.e., to the geometry of Lagrangian subspaces instead of Lagrangian
manifolds, we can only marginally contribute to that program in this paper.

%
%
%

\subsection{Main results and plan of the paper}\label{ss:main-results}
In this paper we shall deal with the preceding \textit{technical} problem. To do that, we
generalize the results of J. Robbin and D. Salamon \cite{RoSa93}, S.E. Cappell, R. Lee, and E.Y.
Miller \cite{CaLeMi94}, K. Furutani, N. Otsuki and the first author in \cite{BoFu99,BoFuOt01} and
of P. Kirk and M. Lesch in \cite{KiLe00}. We give a rigorous definition of the Maslov index for
continuous curves of Fredholm pairs of Lagrangian subspaces in a fixed Banach space with varying
weak symplectic structures and \textit{continuously} varying symplectic splittings and derive its
basic properties. Part of our results will be formulated and proved for relations instead of
operators to admit wider application.

Throughout, we aim for a \textit{clean} presentation in the sense that results are proved in
suitable generality. We wish to show clearly the minimal assumptions needed in order to prove the
various properties. We shall, e.g., prove purely algebraic results algebraically in symplectic
vector spaces and purely topological results in Banach spaces whenever possible - in spite of the
fact that we shall deal with symplectic Hilbert spaces in most applications.

The routes of \cite{BoFu99,BoFuOt01} and \cite{KiLe00} are barred to us because they rely on the
concept of strong symplectic Hilbert space. Consequently, we have to replace some of the familiar
reasoning of symplectic analysis by new arguments. A few of the most elegant lemmata of strong
symplectic analysis can not be retained, but, luckily, the new weak symplectic set-up will show a
considerable strength that is illustrative and applicable also in the conventional strong case.

In Section \ref{s:weak}, we give a thorough presentation of weak symplectic functional analysis.
Basic concepts are defined in Subsection \ref{ss:basic-symplectic}. A new feature of weak
symplectic analysis is the lack of a canonical symplectic splitting: for \textit{strong} symplectic
Hilbert space, we can assume $J^2=-I$ by smooth deformation of the metric, and obtain the canonical
splitting $X=X^+\oplus X^-$ into mutually orthogonal closed subspaces $X^\pm := \ker(J\mp iI)$
which are both invariant under $J$. That permits the representation of all Lagrangian subspaces as
graphs of unitary operators from $X^+$ to $X^-$ (see Lemma \ref{l:strong-symplectic}), which yields
a transfer of contractibility from the unitary group to the space of Lagrangian subspaces.
Moreover, that representation is the basis for a functional analytical definition of the Maslov
index. For \textit{weak} symplectic Hilbert or Banach spaces, the preceding construction does not
work any longer and we must assume that a symplectic splitting is given and fixed (its existence
follows, however, from Zorn's Lemma). Given an elliptic differential operator $A$ of first order
over a manifold $M$ with boundary $\Si$, however, we have a natural symplectic splitting of the
symplectic spaces of sections over $\Si$, both in the strong and weak symplectic case, see Remark
\ref{r:splitting}a, Equation \ref{e:hpm-in-h12}.

In Subsection \ref{ss:fredholm-pairs}, we turn to Fredholm pairs of Lagrangian subspaces to prepare
for the counting of intersection dimensions in the definition of the Maslov index. Here another new
feature of weak symplectic analysis is that the Fredholm index of a Fredholm pair of Lagrangian
subspaces does not need to vanish. On the one hand, this opens the gate to new interesting
theorems. On the other hand, the re-formulation of well-known definitions and lemmata in the weak
symplectic setting becomes rather heavy since we have to add the vanishing of the Fredholm index as
an explicit assumption.

As a side effect of our weak symplectic investigation, we hope to enrich the classical literature
with our new purely algebraic conditions for isotropic subspaces becoming Lagrangians, in Lemma
\ref{l:isotropic-sum-to-lagrangian} and Propositions \ref{p:iso-to-lag} and
\ref{p:fp-characterization}.

At present, the homotopy types of the full Lagrangian Grassmannian and of the Fredholm Lagrangian
Grassmannian remain unknown for weak symplectic structures. We give a list of related open problems
in Subsection \ref{ss:open-problems} below. To us, however, it seems remarkable that a wide range
of familiar geometric features can be re-gained in weak symplectic functional analysis --- in spite
of the incomprehensibility of the basic topology.

In Subsection \ref{ss:isometric}, we lay the next foundation for a rigorous definition of the
Maslov index by investigating continuous curves of operators and relations
that generate Lagrangians in the new wider
setting. Referring to the concepts of our Appendix, we define the spectral flow of such curves.

In Section \ref{s:maslov} we finally come to the intersection geometry. In Subsection
\ref{ss:maslov}, we show how to treat varying weak symplectic structures in a fixed Banach space
with \textit{continuously} varying symplectic splittings and define the Maslov index for continuous
curves of Fredholm pairs of Lagrangian subspaces in this setting. We obtain the full list of basic
properties of the Maslov index as listed by S.E. Cappell, R. Lee, and E.Y. Miller in
\cite{CaLeMi94}. We can not claim that this new Maslov index is always independent of the splitting
projections. However, for strong symplectic Banach space the independence will be proved in
Proposition \ref{p:mas-independence}. That establishes the coincidence with the common definition
of the Maslov index.

In Subsection \ref{ss:comparison2real}, in our general context, we establish the relation between real
symplectic analysis (in the tradition of classical mechanics) on the one side,
and the more elegant complex symplectic analysis (as founded by J. Leray in \cite{Le78}) on the other side.

In Subsection \ref{ss:embedding}, we pay special attention to questions related to the embedding
of symplectic spaces, Lagrangian subspaces and curves into larger symplectic spaces.
Our investigations are inspired by the extremely delicate embedding questions
between the two strong symplectic Hilbert spaces $\beta(A)$ and  $L^2(\Si)$ as studied
by K. Furutani, N. Otsuki and the first author in
\cite{BoFuOt01}. One additional reason for our
interest in embedding problems is our observation of Remark \ref{r:weak-strong}c,
that each weak symplectic Hilbert space can naturally be embedded in a strong symplectic Hilbert space,
imitating the embedding of $H^{1/2}(\Si)$ into $L^2(\Si)$.

In Appendix \ref{ss:gap} and \ref{ss:clr}, we recall the basic knowledge and fix our notations
regarding gaps between closed subspaces in Banach space, uniform properties, closed linear
relations  and their spectral projections. Then, in Appendix \ref{ss:A3}, we give a rigorous
definition of the spectral flow for admissible families of closed relations. Our discussion of
continuous operator families in Subsection \ref{ss:isometric} and the whole of Section
\ref{s:maslov} is based on that definition.

The main results of this paper were achieved many years ago by the authors and informally
disseminated in \cite{BoZh04}. Through all the years, our goal was to establish a truly general
spectral flow formula by applying the weak symplectic functional analysis. But here we met a
technical gap in the argumentation: Only recently we found the correct sufficient conditions for
continuous variation of the Cauchy data spaces (or, alternatively stated, the continuous variation
of the pseudo-differential \Calderon\ projection) for curves of elliptic operators in joint work
with G. Chen and M. Lesch \cite{BCLZ}. Now that gap is bridged, a full general spectral flow
formula is obtained in \cite{BoZh10b} and the relevance of weak symplectic functional analysis has
become sufficiently clear for a regular publication of our results.

%
%
%
%


\section{Weak symplectic functional analysis}\label{s:weak}

\subsection{Basic symplectic functional analysis}\label{ss:basic-symplectic}
We fix our notation. To keep track of the required assumptions, we shall not always assume that the
underlying space is a Hilbert space but permit Banach spaces and --- for some concepts --- even
just vector spaces. For easier presentation and greater generality, we begin with {\em complex}
symplectic spaces.

\begin{definition}\label{d:symplectic-space}
Let $X$ be a complex Banach space. A mapping
\[
  \w\colon X\times X\too \C
\]
is called a ({\em weak}) {\em symplectic form} on $X$, if it is sesquilinear, bounded,
skew-symmetric, and non-degenerate, i.e.,

\noi (i) $\w(x,y)$ is linear in $x$ and conjugate linear in $y$;

\noi (ii) $|\w(x,y)| \leq C\|x\| \|y\|$ for all $x,y\in X$;

\noi (iii) $\w(y,x)\ =\ -\ol{\w(x,y)}$;

\noi (iv) $X^{\w} \ :=\  \{x\in X \mid \w(x,y)\ =\ 0\text{ for all $y\in X$}\} \ =\  \{0\}$.

\noi Then we call $(X,\w)$ a {\em (weak) symplectic Banach space}.
\end{definition}

There is a purely algebraic concept, as well.
\begin{definition}\label{d:algebraic-symplectic-space}
Let $X$ be a complex vector space and $\w$ a form which satisfies
all the assumptions of Definition \ref{d:symplectic-space} except
(ii). Then we call $(X,\w)$ a {\em complex symplectic vector space}.
\end{definition}

\begin{definition}\label{d:lagrangian}
Let $(X,\w)$ be a complex symplectic vector space.

\noi (a) The {\em annihilator} of a subspace ${\la}$ of $X$ is
defined by
\[
{\la}^{\w} \ :=\  \{y\in X \mid \w(x,y)\ =\ 0 \quad\text{ for all $x\in {\la}$}\} .
\]

\noi (b) A subspace ${\la}$ is called {\em symplectic}, {\em
isotropic}, {\em co-isotropic}, or {\em Lagrangian} if
\[
\la\cap{\la}^{\w}\ =\ \{0\}\,,\quad{\la} \,\<\, {\la}^{\w}\,,\quad
{\la}\,\>\, {\la}^{\w}\,,\quad {\la}\,\ =\ \, {\la}^{\w}\,,
\]
respectively.

\noi (c) The {\em Lagrangian Grassmannian} $\Ll(X,\w)$ consists of
all Lagrangian subspaces of $(X,\w)$.
\end{definition}

\begin{definition}\label{d:symplectic-splitting}
Let $(X,\w)$ be a symplectic vector space and $X^+,X^-$ be linear
subspaces. We call $(X,X^+,X^-)$ a {\em symplectic splitting} of
$X$, if $X=X^+ \oplus X^-$, the quadratic form $-i\w$ is positive
definite on $X^+$ and negative definite on $X^-$, and
\begin{equation}\label{e:h-transversality}
\w(x,y)\ =\ 0 \quad\text{ for all $x\in X^+$ and $y\in X^-$}\,.
\end{equation}
\end{definition}

\begin{remark}\label{r:lagrangian}
(a) By definition, each one-dimensional subspace in real symplectic space is isotropic, and there
always exists a Lagrangian subspace. However, there are complex symplectic Hilbert spaces without
any Lagrangian subspace. That is, in particular, the case if $\dim X^+\ne \dim X^-$ in $\NN\cup
\{\infty\}$ for a single (and hence for all)
symplectic splittings.

\noi (b) If $\dim X$ is finite, a subspace ${\la}$ is Lagrangian if
and only if it is isotropic with $\dim {\la} = \12 \dim X$.

\noi (c) In symplectic Banach spaces, the annihilator ${\la}^{\w}$ is closed for any subspace
$\la$. In particular, all Lagrangian subspaces are closed, and we have for any subspace $\la$ the
inclusion
\begin{equation}\label{e:double-annih}
{\la}^{\w\w} \> \ol{\la}.
\end{equation}

\noi (d) Let $X$ be a vector space and denote its (algebraic) dual space by $X'$. Then each
symplectic form $\w$ induces a uniquely defined injective mapping $J\colon X\to X'$ such that
\begin{equation}\label{e:almost-complex}
\w(x,y) \ =\  (Jx,y) \quad\text{ for all $x,y\in X$},
\end{equation}
where we set $(Jx,y):=(Jx)(y)$.

If $(X,\w)$ is a symplectic Banach space, then the induced mapping $J$ is a bounded, injective
mapping $J\colon X\to X^*$ where $X^*$ denotes the (topological) dual space. If $J$ is also
surjective (so, invertible), the pair $(X,\w)$ is called a {\em strong symplectic Banach space}. As
mentioned in the Introduction, we have taken the distinction between {\em weak} and {\em strong}
symplectic structures from Chernoff and Marsden \cite[Section 1.2, pp. 4-5]{ChMa74}.

If $X$ is a Hilbert space with symplectic form $\w$, we identify $X$ and $X^*$. Then the induced
mapping $J$ is a bounded, skew-self-adjoint operator (i.e., $J^* =-J$) on $X$ with $\ker J=\{0\}$
and can be written in the form $J=\begin{pmatrix} iA_+ & 0\\ 0&-iA_-\end{pmatrix}$ with $A_\pm > 0$
bounded self-adjoint (but not necessarily invertible, i.e., $A_\pm\ii$ not necessarily bounded). As
in the strong symplectic case, we then have that $\la\< X$ is Lagrangian if and only if
$\la^\perp=J\la$\,.


\end{remark}


The proof of the following lemma is straightforward and is omitted.

\begin{lemma}\label{l:weak-strong}
Any strong symplectic Hilbert space $(X,\lla\cdot,\cdot\rra,\w)$
(i.e., with invertible $J$) can be made into a strong symplectic
Hilbert space $(X,\lla\cdot,\cdot\rra',\w)$ with $J'^2=-I$ by smooth
deformation of the inner product of $X$ into
\[
\lla x,y\rra'\ :=\  \lla \sqrt{J^*J}x,y\rra
\]
without changing $\w$.
\end{lemma}

\begin{remark}\label{r:weak-strong}
(a) In a strong symplectic Hilbert space many calculations become quite easy. E.g., the inclusion
\eqref{e:double-annih} becomes an equality,  and all Fredholm pairs of Lagrangian subspaces have
vanishing index, see below Definition \ref{d:fredholm-pair}, Equations
\eqref{e:fp-alg}-\eqref{e:fp}.

\noi (b) From the Introduction, we recall an important example of a weak symplectic
Hilbert space: Let $A$ be a formally self-adjoint linear elliptic
differential operators of first order over a smooth compact
Riemannian manifold $M$ with boundary $\Si$. As mentioned in the
Introduction, we have (we suppress mentioning the vector bundle)
\begin{equation}\label{e:symp-honehalf}
H^{1/2}(\Si) \simeq H^1(M)/H^1_0(M)
\end{equation}
with uniformly equivalent norms. Green's form yields a strong
symplectic structure on $L^2(\Si)$ by
\begin{equation}\label{e:symp-l2}
\{x,y\} \ :=\  -\lla Jx,y\rra_{L^2(\Si)}\,.
\end{equation}
Here $J$ denotes the principal symbol of the operator $A$ over the
boundary in inner normal direction. It is invertible since $A$ is
elliptic. For the induced symplectic structure on $H^{1/2}(\Si)$  we
define $J'$ by
\[
\{x,y\}  \ =\  -\lla J'x,y\rra_{H^{1/2}(\Si)}  \quad\text{ for
$x,y\in H^{1/2}(\Si)$}.
\]
Let $B$ be a formally self-adjoint elliptic operator $B$ of first
order on $\Si$. By G{\aa}rding's inequality, the $H^{1/2}$ norm is
equivalent to the induced graph norm. 
This yields $J'=(I+|B|)\ii J$.
Since $B$ is elliptic, it has compact resolvent. So, $(I+|B|)\ii$ is
compact in $L^2(\Si)$; and so is $J'$. Hence $J'$ is not invertible.
In the same way, any dense subspace of $L^2(\Si)$ inherits a weak
symplectic structure from $L^2(\Si)$.

\noi (c) Each weak symplectic Hilbert space $(X,\lla\cdot,\cdot\rra,\w)$ with induced injective
skew-self-adjoint $J$ can naturally be embedded in a strong symplectic Hilbert space
$\bigl(X',\lla\cdot,\cdot\rra',\w '\bigr)$ with invertible induced $J'$ by setting $\lla x,y\rra
':= \lla|J| x,y\rra$ as in Lemma \ref{l:weak-strong} and then completing the space. This imitates
the situation of the embedding of $H^{1/2}(\Si)$ into $L^2(\Si)$\,. It shows that the weak
symplectic Hilbert space $H^{1/2}(\Si)$ with its embedding into $L^2(\Si)$ yields a model for all
weak symplectic Hilbert spaces. In Section \ref{ss:embedding}, we shall elaborate on the embedding
weak $\into$ strong a little further.

\end{remark}

\medskip
The following lemma is a key result in symplectic analysis. The representation of Lagrangian
subspaces as graphs of unitary mappings from one component $X^+$ to the complementary component
$X^-$ of the underlying symplectic vector space (to be considered as the induced complex space in
classical real symplectic analysis, see, e.g., K. Furutani and the first author \cite[Section
1.1]{BoFu98}) goes back to J. Leray \cite{Le78}. We give a simplification for complex vector
spaces, first announced in \cite{Zh01}. Of course, the main ideas were already contained in the
real case. The Lemma is essentially well-known and will be obtained in the more general setting
below: (i) is clear; (ii) will follow from Lemma \ref{l:lagrangian-representation}; and (iii) from
Proposition \ref{p:fp-characterization}.

\begin{lemma}\label{l:strong-symplectic}
Let $(X,\w)$ be a strong symplectic Hilbert space with $J^2=-I$.
Then
\begin{enumerate}
\item
the space $X$ splits into the direct sum of mutually orthogonal
closed subspaces
\[
X\ =\ \ker (J-iI)\oplus\ker(J+iI),
\]
which are both invariant under $J$;

\item there is a 1-1 correspondence between the space $\Uu^J$ of unitary
operators from $\ker(J-iI)$ to $\ker(J+iI)$ and $\Ll(X,\w)$ under the mapping $U\mapsto {\la}:=
\GGG(U)$ (= graph of $U$);

\item if $U,{V}\in\Uu^J$ and ${\la}:=\GGG(U)$, $\mu:=\GGG({V})$, then
$({\la},\mu)$ is a Fredholm pair (see Definition
\ref{d:fredholm-pair}b)
 if and only if $U-V$, or, equivalently,
$UV\ii-I_{\ker(J+iI)}$ is Fredholm. Moreover, we have a natural
isomorphism
\begin{equation}\label{e:unitary-counting}
\ker(UV\ii-I_{\ker(J+iI)}) \simeq {\la}\cap \mu\,.
\end{equation}
\end{enumerate}

\end{lemma}


The preceding method to characterize Lagrangian subspaces and to
determine the dimension of the intersection of a Fredholm pair of
Lagrangian subspaces provides the basis for defining the Maslov
index in strong symplectic spaces of infinite dimensions (see, in
different formulations and different settings, the quoted references
\cite{BoFu98}, \cite{BoFuOt01}, \cite{FuOt02}, \cite{KiLe00}, and
Zhu and Long \cite{ZhLo99}).

Surprisingly, it can be generalized to weak symplectic Banach spaces
in the following way.

\begin{lemma}\label{l:lagrangian-representation}
Let $(X,\w)$ be a symplectic vector space with a symplectic
splitting $(X,X^+,X^-)$.

\noi (a) Each isotropic subspace ${\la}$ can be written as the graph
\[
{\la}\ =\ \GGG(U)
\]
of a uniquely determined injective operator
\[
U\colon \dom(U) \too X^-
\]
with $\dom(U) \< X^+$\,. Moreover, we have
\begin{equation}\label{e:almost-unitary}
\w(x,y) \ =\  -\w(Ux,Uy) \quad\text{ for all $x,y\in\dom(U)$}.
\end{equation}

\noi (b) If $X$ is a Banach space, then $X^{\pm}$ are always closed and the operator $U$ defined by
a Lagrangian subspace $\la$ is closed as an operator from $X^+$ to $X^-$ (not necessarily densely
defined).

\noi (c) For a closed isotropic subspace $\la$ in a strong symplectic Banach space $X$, we have
$\dom(U)$ and $\ran U$ are closed. Moreover, if $\la$ is  Lagrangian, then $\dom(U)=X^+$ and $\ran
U =X^-$; i.e., the generating $U$ is bounded and surjective with bounded inverse.
\end{lemma}

\begin{proof} {\it a}. Let $\la\< X$ be isotropic and $v_++v_-, w_++w_- \in {\la}$
with $v_\pm,w_\pm\in X^\pm$\,. By the isotropic property of ${\la}$
and our assumption about the splitting $X=X^+\oplus X^-$ we have
\begin{equation}\label{e:symplectic-splitting}
0\ =\ \w(v_++v_-,w_++w_-)\ =\ \w(v_+,w_+)+\w(v_-,w_-).
\end{equation}
In particular, we have
\[
\w(v_++v_-,v_++v_-) \ =\  \w(v_+,v_+)+\w(v_-,v_-)\ =\ 0
\]
and so $v_-=0$ if and only if $v_+=0$. So, if the first (respectively the
second) components of two points $v_++v_-,w_++w_-\in {\la}$
coincide, then also the second (respectively the first) components must
coincide.

Now we set
\[
\dom(U) \ :=\  \{x\in X^+\mid \exists{y\in X^-} \text{ such that $x+y\in \la$}\}.
\]
By the preceding argument, $y$ is uniquely determined, and we can define $Ux:=y$. By construction,
the operator $U$ is an injective linear mapping, and property \eqref{e:almost-unitary} follows from
\eqref{e:symplectic-splitting}.

\noi {\it b}. By Definition \ref{d:symplectic-splitting} of a symplectic splitting, Equation
\eqref{e:h-transversality} we have $X^-\< (X^+)^\w$\,. Now let $x_++x_-\in (X^+)^{\w}$ with
$x_{\pm}\in X^{\pm}$. Then $\w(x_++x_-,x_+)=\w(x_+,x_+)=0 \iff x_+=0$ since $-i\w$ is positive
definite on $X^+$. That proves $X^-=(X^+)^{\w}$, and correspondingly $X^+=(X^-)^{\w}$. As noticed
in Remark \ref{r:lagrangian}c, annihilators are always closed. This proves the first part of ({\it
b}). Now let ${\la}$ be a Lagrangian subspace and let $U$ be the uniquely determined injective
operator $U\colon \dom(U)\to X^-$ with $\dom(U)\< X^+$ and $\GGG(U)=\la$. By Definition
\ref{d:lagrangian}b we have ${\la}={\la}^\w$, hence ${\la}$ is closed as an annihilator and so is
the graph of $U$, i.e., $U$ is closed.

\noi {\it c}. Let $\la=\GGG(U)$. Let $\{x_n\}$ be a sequence in $\dom (U)$ convergent to $x\in
X^+$. Since $X$ is strong, we see from \eqref{e:almost-unitary} that the sequence $\{Ux_n\}$ is a
Cauchy sequence and therefore is also convergent. Denote by $y$ the limit of $\{Ux_n\}$. Since
$\la$ is closed, we have $x\in\dom U$ and $y=Ux$. Thus $\dom(U)$ is closed. We apply the same
argument to $\dom(U\ii)\<X^-$, relative to the inner product $i\w$ and obtain that $\ran U$ is
closed. This proves the first part of ({\it c}).

Now assume that ${\la}$ is a Lagrangian subspace. Firstly we show that $U$ is densely defined in
$X^+$\,. Indeed, if $\ol{\dom (U)} \neq X^+$\/, there would be a $v\in V$, $v\ne 0$, where $V$
denotes the orthogonal complement of $\dom (U)$ in $X^+$ with respect to the inner product on $X^+$
defined by $-i\w$. Clearly $(\dom (U))^\w =V+X^-$\,. So, $V=(\dom (U))^\w \cap X^+$\,. Then $v+0\in
{\la}^\w\setminus {\la}$. That contradicts the Lagrangian property of ${\la}$. So, we have
$\ol{\dom (U)} = X^+$\,.

We have shown that $\dom(U)$ is closed and dense. Hence $\dom(U)=X^+$. Now the boundedness of $U$
follows from the closedness of $\GGG(U)$. Applying the same arguments to $\dom(U\ii)\<X^-$ relative
to the inner product $i\w$ yields $\ran U =\dom(U\ii)= X^-$ and $U\ii$ is bounded.
\end{proof}

\begin{remark}\label{r:splitting}
(a) Note that the symplectic splitting is not unique. Its existence can be proved by Zorn's Lemma.
In our applications, the geometric background provides natural splittings. Let $A$ be an elliptic
differential operator of first order, acting on sections of a Hermitian vector bundle $E$ over the
Riemannian manifold $M$ with boundary $\Si$. Then the symplectic Hilbert space structures of
$L^2(\Si;E|_{\Si})$ and $H^{1/2}(\Si;E|_{\Si})$ of \eqref{e:symp-l2} and \eqref{e:symp-honehalf}
are compatible and their symplectic splitting is defined by the bundle endomorphism (the principal
symbol of $A$ in inner normal direction) $J\colon E|_{\Si}\to E|_{\Si}$ in the following way:
\begin{multline}\label{e:hpm-in-h12}
H^\pm \ :=\  H^{1/2}(\Si;E^\pm|_{\Si}) \quad\tand\quad L^\pm \ :=\ L^2(\Si;E^\pm|_{\Si})
\\ \text{ with }
E^\pm|_{\Si}\ :=\  \text{ lin. span of}\left\{\begin{array}{l}\text{positive}\\ \text{negative}
\end{array} \right\} \text{ eigenspaces of $iJ$}.
\end{multline}
Note that $L^+, L^-$ change continuously if $J$ changes continuously. For varying splittings see
also the discussion below in Section \ref{s:maslov}.

\noi (b) The symplectic splitting and the corresponding {\em graph}
representation of isotropic and Lagrangian subspaces must be
distinguished from the splitting in complementary Lagrangian
subspaces which yields the common representation of Lagrangian
subspaces as {\em images} in the real category (see Lemma
\ref{l:mas-bofu=mas-zhu} below).
\end{remark}

\subsection{Fredholm pairs of Lagrangian subspaces}\label{ss:fredholm-pairs}
A main feature of symplectic analysis is the study of the {\em
Maslov index}. It is an intersection index between a path of
Lagrangian subspaces with the {\em Maslov cycle}, or, more
generally, with another path of Lagrangian subspaces.

Before giving a rigorous definition of the Maslov index in weak
symplectic functional analysis (see below Section
\ref{s:maslov}) we fix the terminology and give several simple
criteria for a pair of isotropic subspaces to be Lagrangian.

We recall:

\begin{definition}\label{d:fredholm-pair}
(a) The space of (algebraic) \emph{Fredholm pairs} of linear
subspaces of a vector space $X$ is defined by
\begin{equation}\label{e:fp-alg}
\Ff^{2}_{\operatorname{alg}}(X)\ :=\ \{({\lambda},{\mu})\mid  \dim \lambda\cap\mu  <+\infty \text{
and $\dim X/(\lambda+\mu)<+\infty$}\}
\end{equation}
with
\begin{equation}\label{e:fp-index}
\index(\la,\mu)\ :=\ \dim\la\cap\mu - \dim X/(\la+\mu).
\end{equation}

\noi (b) In a Banach space $X$, the space of (topological)
\emph{Fredholm pairs} is defined by
\begin{equation}\label{e:fp}
\Ff^{2}(X)\ :=\ \{({\lambda},{\mu})\in\Ff^2_{\operatorname{alg}}(X)\mid {\lambda},{\mu}
\tand{\lambda}+{\mu} \subset X \text{ closed}\}.
\end{equation}
\end{definition}

\begin{remark}\label{r:redholm-pairs}
Actually, in Banach spaces the closedness of $\la+\mu$ follows from its finite codimension in $X$
in combination with the closedness of $\la,\mu$ (see \cite[Remark A.1]{BoFu99} and \cite[Problem
4.4.7]{Ka76}). So, the set of algebraic Fredholm pairs of Lagrangian subspaces of a symplectic
Banach space $X$ coincides with the set $\Ff\Ll^2(X)$ of topological Fredholm pairs of Lagrangian
subspaces of $X$.
\end{remark}

We begin with a simple algebraic observation.

\begin{lemma}\label{l:isotropic-sum-to-lagrangian}
Let $(X,\w)$ be a symplectic vector space with transversal subspaces
$\la,\mu$\,. If $\la,\mu$ are isotropic subspaces, then they are
Lagrangian subspaces.
\end{lemma}

\begin{proof}
From linear algebra we have
\[
\la^{\w}\cap \mu^{\w} \ =\ (\la+\mu)^{\w} \ =\ \{0\},
\]
since $\la+\mu=X$. From
\begin{equation}\label{e:isotropics}
\la\<\la^{\w}, \mu\<\mu^{\w}
\end{equation}
we get
\begin{equation}\label{e:new-splitting}
X\ =\  \la^{\w} \oplus \mu^{\w}\,.
\end{equation}
To prove $\la^{\w}=\la$ (and similarly for $\mu$), we consider an $x\in\la^{\w}$\,. It can be
written in the form $x=y+z$ with $y\in\la$ and $z\in \mu$ because of the splitting $X=
\la\oplus\mu$. Applying \eqref{e:isotropics} and the splitting \eqref{e:new-splitting} we get $y=x$
and so $z=0$, hence $x\in\la$.
\end{proof}

With a little work, the preceding lemma can be generalized from direct sum decompositions to
(algebraic) Fredholm pairs. At first we have

\begin{lemma}\label{l:compare-dim} Let $V,W$ be two vector spaces
and $f\colon V\times W\to\C$ be a sesquilinear mapping. Assume that $\dim W<+\infty$. If for each
$v\in V$, the condition $f(v,w)=0$ for all $w\in W$ implies $v=0$, then we have $\dim V\,\le\,\dim
W$.
\end{lemma}

\begin{proof} Let $\wt W$ be the space of conjugate linear functionals on $W$.
Let $\wt f\colon V\to\wt W$ be the induced map of $f$ defined by $(\wt f(v))(w):=f(v,w)$. Then $\wt
f$ is linear. Our condition implies that $\wt f$ is injective. Thus we have $\dim V\,\le\,\dim\wt
W\,=\,\dim W$.
\end{proof}

\begin{corollary}\label{c:isotropic-index} Let $(X,\w)$ denote a symplectic vector
space.

\noi (a) For any finite-codimensional linear subspace $\la$, we have
$\dim\la^{\w}\,\le\,\dim X/\la$.

\noi (b) For any finite dimensional linear subspace $\mu$, we have
$\mu^{\w\w}=\mu$ and $\dim\mu=\dim X/\mu^{\w}$.
\end{corollary}

\begin{proof} {\it a}. Define $f\colon \la^{\w}\times (X/\la)\to\C$ by
$f(x,y+\la):=\omega(x,y)$ for all $x\in\la^{\w}$ and $y\in X$. Then $f$ satisfies the condition in
Lemma \ref{l:compare-dim}. So our result follows.

\noi {\it b}. Define $g\colon (X/\mu^{\w})\times \mu\to\C$ by
$g(x+\mu^{\w},y):=\overline{\omega(x,y)}$ for all $x,y\in\mu$. Then $g$ satisfies the condition in
Lemma \ref{l:compare-dim}. So we have $\dim X/\mu^{\w}\le\dim\mu$. By (a) we have
$\dim\mu^{\w\w}\le\dim X/\mu^{\w}$. Since $\mu\subset\mu^{\w\w}$, our result follows.
\end{proof}

\begin{proposition}\label{p:iso-to-lag}
Let $(X,\w)$ be a symplectic vector space and $(\la,\mu) \in
\Ff^2_{\operatorname{alg}}(X)$\,. If $\la,\mu$ are isotropic
subspaces with $\index(\la,\mu)\geq 0$, then $\la$ and $\mu$ are
Lagrangian subspaces of $X$,
\[
\index(\la,\mu)\ \fequal{(i)}\  0,\quad (\la+\mu)^{\w}\  \fequal{(ii)}\ \la\,\cap\,\mu,\quad\tand
(\la+\mu)^{\w\w}\  \fequal{(iii)}\ \la+\mu.
\]
\end{proposition}

\begin{proof} Set $\wt{X}:= (\la+\mu)/(\la\cap\mu)$ with the induced form
\[
\wt{\w}([x+y],[\xi+\eta])\ :=\  \w(x+y,\xi+\eta) \quad\text{ for $x,\xi\in\la \quad\tand\quad
y,\eta\in \mu$},
\]
where $[x+y]:= x+y+\la\cap\mu$ denotes the class of $x+y$ in $\frac {\la+\mu}{\la\cap\mu}$\,. The
aim is to show that $\wt X$  is a symplectic vector space. During the proof of this fact the
claimed equalities (i)-(iii) will be obtained.

Since $\la,\mu$ are isotropic, we have $\w(x+y+z,\xi+\eta+\z)=\w(x+y,\xi+\eta)$ for any $z,\z\in
\la\cap\mu$. So $\wt{\w}$ is well-defined and inherits the algebraic properties from $\w$.

To show that $(\wt{X})^{\wt{\w}}=\{0\}$, we observe
\begin{equation}\label{e:symp-alg}
(\la+\mu)^{\w}\ =\ \la^{\w}\cap\mu^{\w} \> \la\cap\mu\,.
\end{equation}
By Corollary \ref{c:isotropic-index}a, we have
\[ \dim
(\la+\mu)^{\w}\le\dim  X/(\la+\mu) \leq \dim(\la\cap\mu).
\]
Here the last inequality is just the non-negativity of the Fredholm index as defined in
\eqref{e:fp-index}. This proves (i), namely
\begin{equation}\label{e:symp-alg-mod}
\dim(\la+\mu)^{\w}\ =\ \dim X/(\la+\mu)\ =\ \dim(\la\cap\mu).
\end{equation}
Combining \eqref{e:symp-alg-mod} with \eqref{e:symp-alg} yields (ii), namely
\begin{equation}\label{e:symp-alg-modd}
\la\cap\mu\ =\ \la^{\w}\cap\mu^{\w}\ =\ (\la+\mu)^{\w}.
\end{equation}
By Corollary \ref{c:isotropic-index}b, we have
\[\dim X/(\la+\mu)\ =\ \dim\la\cap\mu\ =\ \dim X/(\la\cap\mu)^{\w}\ =\ \dim X/(\la+\mu)^{\w\w}.\]
Thus we have proved (iii), namely $\la+\mu=(\la+\mu)^{\w\w}$.

To finish our proof that $\wt\w$ is non-degenerate, one checks that
\begin{equation}
\Bigl(\frac{\la+\mu}{\la\cap\mu}\Bigr)^{\wt{\w}} \ =\
\frac{(\la+\mu)^{\w}}{\la\cap\mu}\,.
\end{equation}
With \eqref{e:symp-alg-modd} that proves that $\displaystyle
\bigl(\frac{\la+\mu}{\la\cap\mu}\bigr)^{\w}=\{0\}$, hence $\displaystyle \wt X=
\frac{\la+\mu}{\la\cap\mu}$ is a true symplectic vector space for the induced form $\wt{\w}$. It is
spanned by the transversal isotropic subspaces
\[
\frac{\la+\mu}{\la\cap\mu} \ =\  \frac{\la}{\la\cap\mu} \oplus
\frac{\mu}{\la\cap\mu}\,.
\]
By Lemma \ref{l:isotropic-sum-to-lagrangian}, the spaces $\displaystyle \frac{\la}{\la\cap\mu}\,,
\frac{\mu}{\la\cap\mu}$ are Lagrangian subspaces.

It remains to prove that $\la,\mu$ itself are Lagrangian subspaces of $X$. Clearly
$\la\<\la^{\w}\cap(\la+\mu)$. Now consider $x\in\la$ and $y\in\mu$ with $x+y\in\la^{\w}$\,. Then
\[
[x+y] \in \Bigl(\frac{\la}{\la\cap\mu}\Bigr)^{\wt{\w}} \ =\
\frac{\la}{\la\cap\mu}
\]
by the Lagrangian property of $\frac{\la}{\la\cap\mu}$\,. It follows
that $x+y\in\la$, hence
\begin{equation}\label{e:almost-lagrangian}
\la^{\w}\cap(\la+\mu) \ =\  \la \text{ and similarly }
\mu^{\w}\cap(\la+\mu) \ =\  \mu\,.
\end{equation}
Combined with the fact that
\[
\la^{\w} \< (\la\cap\mu)^{\w} \ =\ (\la+\mu)^{\w\w}\ =\  \la +\mu,
\]
the inclusion $\la\>\la^{\w}$ follows and so the Lagrangian property
of $\la$ (and similarly of $\mu$).
\end{proof}

\begin{remark}\label{r:open-problems}
For related topological (unsolved) questions see below Subsection \ref{ss:open-problems}.
\end{remark}

We close this subsection with the following characterization of
Fredholm pairs.

\begin{proposition}\label{p:fp-characterization}
Let $(X,\w)$ be a symplectic Banach space  and let $(X,X^+,X^-)$ be a symplectic splitting. Let
$\la,\mu$ be isotropic subspaces. Let $U,V$ denote the generating operators for $\la,\mu$ in the
sense of Lemma \ref{l:lagrangian-representation}. We assume that%
\begin{equation}\label{e:v-bounded-bounded-invertible}
V\colon X^+\to X^-\quad\text{is bounded and bounded invertible}.
\end{equation}
Then \newline(a) The space $\mu$ is a Lagrangian subspace of $X$.
\newline (b) Moreover,
\begin{equation*}
(\la,\mu)\in \Ff^2(X) \iff UV\ii-I_{X^-} \text{ is a Fredholm operator with domain $\,V(\dom\,
U)$}.
\end{equation*}
\newline (c) In this case, $U-V$ is a (closed, not necessarily bounded) Fredholm operator with domain
$\dom\, U$ and
\[
\index(\la,\mu)\ =\ \index(UV\ii -I_{X^-}).
\]
\newline (d) In particular, $U-V$ (and thus $UV^{-1}-I_{X^-}$) is closed if $\lambda$ is closed and $V$ is bounded (as assumed above).
\end{proposition}

\begin{note} Our assumption \eqref{e:v-bounded-bounded-invertible} is needed for (a). For
(b) and (c) (in \eqref{e:la-plus-mu} below) it is only required that $\dom(U)\subset\dom(V)$. For
(d) we need only that $V$ is bounded.
\newline For (b), (c) and (d) recall from Definition \ref{d:fredholm-pair}b) that we require of a pair in
$\Ff^2(X)$ to consist of \textit{closed} subspaces.
\end{note}

\begin{proof} {\it a}. Since $\mu=\GGG(V)$ is an isotropic subspace of $X$,
the space $\mu':=\GGG(-V)$ is also
isotropic. We show that $\mu,\mu'$ are transversal in $X$. Then by Lemma
\ref{l:isotropic-sum-to-lagrangian}, $\mu$ (and $\mu'$) are Lagrangian subspaces. First, from the
injectivity of $V$, we have
$\mu\cap\mu'=\{0\}$.

Next, let $x+y$, or, more suggestively, $\begin{pmatrix} x\\
y\end{pmatrix}$ denote an arbitrary point in $X$ with $x\in X^+$ and $y\in X^-$\,. Since $V$ is
bounded with bounded inverse, we have $y\in \ran V$ and $z,w\in\dom V$, where
\[
z \ :=\  \frac {x+V\ii y}2 \ \tand\ w\ :=\  \frac{x-V\ii y}2\,.
\]
Then $z+w=x$ and $z-w=V\ii y$, so
\[
\begin{pmatrix} x\\ y\end{pmatrix}
\ =\  \begin{pmatrix} z\\ Vz \end{pmatrix} + \begin{pmatrix} w\\
-Vw\end{pmatrix}\,.
\]
This proves $X=\mu\oplus\mu'$\,.

\noi {\it b and c}. Let $\la=\GGG(U)$ and $\mu=\GGG(V)$ with $V$ bounded and bounded invertible.
Let $P_{\pm}$ denote the projections of $X=X^+\oplus X^-$ onto $X^{\pm}$. Then
\[
\la\cap\mu \ =\ \Bigl\{\begin{pmatrix}x \\
Vx\end{pmatrix}\mid x\in \dom(U) \,\tand\, Ux\ =\ Vx\Bigr\}.
\]
So, $P_-$ induces an algebraic and topological isomorphism between
$\la\cap\mu$ and $\ker \bigl(UV\ii - I_{X^-}\bigr)$.

Now we determine
\begin{equation}\label{e:la-plus-mu}
\begin{split}
\la+\mu & \ =\ \Bigl\{\begin{pmatrix}x \\ Ux\end{pmatrix} +
\begin{pmatrix}y
\\ Vy\end{pmatrix} \mid x\in\dom(U),y\in X^+\Bigr\}\\
&\ =\ \Bigl\{\begin{pmatrix}x' \\
Vx'\end{pmatrix} + \begin{pmatrix}0 \\ z\end{pmatrix} \mid x'\in
X^+\,\tand\, z\in \ran(UV\ii-I_{X^-})\Bigr\}\\
&\ =\  \mu \oplus\ran(UV\ii-I_{X^-}).
\end{split}
\end{equation}
The last direct sum sign comes from the invertibility of $V$: It induces $\mu\cap X^-=\{0\}$ and,
similarly, $\mu+X^-= X$. From that we obtain the direct sum decomposition $X=\mu\oplus X^-$ with
projections $\Pi_{\mu}$ and $\Pi_-$ onto the components. So, $\Pi_-$ yields an algebraic and
topological isomorphism of $\la+\mu$ onto $\ran(UV\ii-I_{X^-})$. In particular, we have $\la+\mu$
closed in $X$ if and only if $\ran(UV\ii-I_{X^-})$ is closed in $X^-$\, and
\[
X/(\la+\mu) \simeq X^-/\ran(UV\ii-I_{X^-})
\]
with coincidence of the codimensions.

\noi {\it d}. Let $x_n\to x$ and $y_n\to y$ be such that $(U-V)x_n\to y$. Since $V$ is bounded,
$Vx_n\to Vx$. Then $Ux_n\to Vx+y$. Since $\lambda$ is closed, $U$ is closed. Thus $x\in\dom(U)$ and
$Ux=Vx+y$. Hence $(U-V)x=y$.
\end{proof}


\subsection{Open topological problems}\label{ss:open-problems}

\subsubsection{Fredholm pairs of Lagrangian subspaces with negative index?}
Proposition \ref{p:iso-to-lag} shows that Fredholm pairs of Lagrangian subspaces in symplectic
vector spaces cannot have positive index. In contrast to the strong case, one may expect that we
have pairs with negative index in weak symplectic Hilbert space. By now, however, this is an open
problem.

\subsubsection{Characterization of Lagrangian subspaces by the canonical
symmetry property of the projections?}
 The delicacy of Lagrangian analysis in weak symplectic
Hilbert space may also be illuminated by addressing the orthogonal projection onto a Lagrangian
subspace. In a strong symplectic Hilbert space with unitary $J$, the range of an orthogonal
projection is Lagrangian if and only if the projections $P$ and $I-P$ are conjugated by the
operator $J$ in the way
\[
I-P\ =\ JPJ^*\,,
\]
which is familiar from characterizing elliptic self-adjoint
pseudo-differential boundary conditions for elliptic differential of
first order, see \cite[Proposition 20.3]{BoWo93}. In weak symplectic
analysis, $J$ maps the range $\range P$ onto a dense subset of $\ker
P$, but there the argument stops.

\subsubsection{Contractibility of the space of Lagrangian subspaces?}
 There are two further differences between the weak and
the strong case, namely regarding the topology: while the Lagrangian Grassmannian $\Ll(X,\w)$
inherits contractibility from the space of unitary operators in separable Hilbert spaces by Lemma
\ref{l:strong-symplectic}(ii), more refined arguments will be needed to prove the contractibility
in the weak case, if it is true at all.

\subsubsection{Bott periodicity of the homotopy groups of the space of Fredholm
pairs of Lagrangian subspaces?}
Next, consider the space $\Ff\Ll_{\la}(X)$ of all
Lagrangian subspaces which form a Fredholm pair with a given
Lagrangian subspace $\la$\,. Its topology is presently also unknown in
the weak case, whereas we have
\[
\pi_1\bigl(\Ff\Ll_{\la}(X)\bigr)\cong \Z
\]
in strong symplectic Hilbert spaces $X$ (see \cite[Corollary 4.3]{BoFu99} and the generalization to
Bott periodicity in \cite[Equation (6.2) with Lemma 6.1 and Proposition 6.5]{KiLe00}).

\subsection{Curves of \textit{unitary} operators that are \textit{admissible} with respect to the positive half-line}
\label{ss:isometric} We begin with some observations on inner product spaces and refer to the
Appendix \ref{ss:clr} for a rigorous definition of the basic concepts of linear relations.

\begin{lemma}\label{l:h-unitary-kernel-r} Let $(X,h_X)$, $(Y,h_Y)$, $(Z,h_Z)$
denote three inner product spaces, $A$, $B$ linear relations between
$X$ and $Y$, and $C$ a linear relation between $X$ and $Z$.

\noi (a) Assume that $C$ is a linear operator, $\dom(A)\<\dom(C)$,
and $h_Y(y,y)\le h_Z(Cx,Cx)$ for all $(x,y)\in A$. Then $A$ is a
linear operator.

\noi (b) Assume that $B$ is a linear operator, $\dom(A)=\dom(C)\subset\dom(B)$, and
\begin{equation}\label{e:h-unitary-kernel-c}h_Y(y,y)+h_Z(z,z)\le h_Y(Bx,Bx)
\end{equation}
for all $(x,y)\in A$ and $(x,z)\in C$. Then $A$ and $C$ are linear
operators and $\ker(B-A)\subset\ker C$.
\end{lemma}

\begin{proof} {\it a}. Let $y\in\ker A$, i.e., $(0,y)\in A$. By our assumption we have
$h_Y(y,y)\le h_Z(C0,C0)=0$. Since $h_Y$ is positive definite, we
have $y=0$.

\noi {\it b}. By (a) $A$ and $C$ are linear operators. Let $x\in\ker(B-A)$. Then $Bx=Ax$. By
(\ref{e:h-unitary-kernel-c}) we have $h_Z(Cx,Cx)\le 0$. Since $h_Z$ is positive definite, we have
$Cx=0$, i.e., $x\in\ker C$.
\end{proof}

Let $X$ be a complex Banach space. We apply the following notations:
\[%
\begin{array}
[c]{rl}%
\mathcal{C}(X)\ :=\  & \text{closed operators on $X$ with dense domain},\\
\mathcal{B}(X)\ :=\  & \text{bounded linear operators $X\to X$},\tand\\
\mathcal{CF}(X)\ :=\  & \text{closed (not necessarily bounded) Fredholm operators on $X$}.
\end{array}
\]
The topology on $\mathcal{CF}(X)$ is defined in the Appendix.

We assume that $X$ is an inner product space with a fixed inner product (i.e., a sesquilinear,
symmetric positive definite form) $h\colon X\times X\to\C$ which is bounded
\[
|h(x,y)|\ \leq\  c\norm{x}\norm{y}\quad\text{ for all $x,y\in X$}.
\]

\begin{definition}\label{d:h-unitary}
An operator $A\in\Cc(X)$ will be called {\em unitary} with respect
to $h$, if
\[
h(Ax,Ay)\ =\ h(x,y) \quad\text{ for all $x,y\in \dom (A)$}.
\]
\end{definition}

\begin{remark}\label{r:h-unitary}
(a) Note that $h$ induces a uniformly smaller norm than $\norm{\cdot}$ on $X$ which makes $X$ into
a Hilbert space if and only if $X$ becomes complete for this $h$-induced norm.

\noi (b) The concept of $h$-unitary extends trivially to closed operators with dense domain in one
Banach space equipped with an inner product, and range in a second Banach space, possibly with a
different inner product. In this sense, for any Lagrangian subspace the generating operator
$U\in\Cc(X^+,X^-)$ (established in Lemma \ref{l:lagrangian-representation}) is $(h^+,h^-)$-unitary
with $h^{\pm}=\mp i\omega|_{X^{\pm}}\/$.
\end{remark}

Like for unitary operators in Hilbert spaces, the following lemma shows that a unitary operator
with respect to $h$ has no eigenvalues outside the unit circle.

\begin{lemma}\label{l:spectrum-unitary}
Let $A\in\Cc(X)$ be unitary with respect to $h$ and $\la\in\C$,
$|\la|\neq 1$. Then $\ker(A-\la I)=\{0\}$.
\end{lemma}

\begin{proof}
Let $x\in\ker(A-\la I)$, so $Ax=\la x$ and
\[
h(x,x)\ =\ h(Ax,Ax)\ =\ |\la|^2h(x,x).
\]
Since $|\la|\neq 1$, we get $h(x,x)=0$ and so $x=0$ since $h$ is positive definite.
\end{proof}

For a certain subclass of unitary operators with respect to $h$ we show that they have discrete
spectrum close to 1. Consequently, they are admissible with respect to the positive half-line
$\ell$ (in the sense of Definition \ref{d:admissible} of Appendix \ref{ss:A3}) and so permit the
definition of spectral flow through $\ell$ for continuous families (Appendix \ref{ss:A3}). Here the
co-orientation of $\ell$ is upward.

\begin{proposition}\label{p:unitary-admissible}
(a) Let $X$ be a Banach space with bounded inner product $h$. Let $A\in\Cc(X)$ be $h$-bounded,
i.e., an operator satisfying
\[ h(Ax,Ay)\ \le\  h(x,y) \quad\text{ for all $x,y\in \dom(A)$}.\]
We assume $A-I\in\Cc\Ff(X)$ of index 0. If either $A$ is $h$-unitary or $A$ is bounded, then there
is a bounded neighborhood $N\<\C$ of $1$ with closure $\bar N$ such that
\[
\s(A)\cap \bar N \< \{1\},\quad\dim P_N(A)\ =\ \dim\ker(A-I).
\]

\noi (b) Let $\mathfrak{m}$ be an open submanifold of $\ell:=(0,+\infty)$ and $A$ be unitary with
respect to $h$. If $A$ is admissible with respect to $\mathfrak{m}$ in the sense of Definition
\ref{d:admissible}a of the Appendix \ref{ss:A3}, then $\sigma(A)\cap \mathfrak{m}\subset\{1\}$.

\noi (c) Let $\{h_s\}_{0\le s\le 1}$ be a family of inner products on $X$. Let $A_s\in \Cc(X)$ be
unitary with respect to $h_s$. We assume that the family $\{A_s\}$ is continuous. We denote $h_0=:
h$ and $A_0=:A$ and choose $N$ like in (a). Then for $s\ll 1$ the spectrum part $\s(A_s)\cap N$
consists of eigenvalues of finite algebraic multiplicity and we have
\[
\s(A_s)\cap N \< S^1\,.
\]

\noi (d) Let $\{h_s\}$ and $A_s$ be as in (c). Then there exists $\e\in (0,1)$ such that the family
$\{A_s\}$ is a family of admissible operators that is spectral continuous near
$\mathfrak{m}_{\e}:=(1-\e,1+\e)$ in the sense of Definitions \ref{d:admissible}a and
\ref{d:spectral-continuous}a.

\noi (e) Let $\{h_s\}$ and $A_s$ be as in (c). Let $\mathfrak{m}\ni 1$ be an open submanifold of
$\ell=(0,+\infty)$. If the family $\{A_s\}$ is a family of admissible operators that is spectral
continuous near $\mathfrak{m}$, we can define the spectral flow
\[
\SF_{\ell}\{A_s\}:=\SF_{\mathfrak{m}}\{A_s\}.
\]

\end{proposition}

\begin{proof} {\it a}. Since $\ker(A-I)$ is finite-dimensional, we have an $h$-orthogonal
splitting
\[
X\ =\ \ker(A-I)\oplus X_1
\]
with closed $X_1$\,. (Take $X_1:=\Pi(X)$ with $\Pi(x):=x-\sum_{j=1}^n h(x,e_j)e_j$, where $\{e_j\}$
is an $h$-orthonormal basis of $\ker(A-I)$). We notice that $\ker(A-I)\<\dom(A)$, so
\begin{equation}\label{e:h-orthogonal-domain}
\dom (A)\ =\  \ker(A-I) \oplus (\dom (A)\cap X_1) .
\end{equation}
Then the operator $A$ can be written in block form
\begin{equation}\label{e:a-block-form}
A\ =\  \begin{pmatrix} I_0&A_{01}\\0&A_{11}\end{pmatrix}\,,
\end{equation}
where $I_0$ denotes the identity operator on $\ker(A-I)$.

Since $A$ is $h$-bounded, by Lemma \ref{l:h-unitary-kernel-r}b we have $\ker(A_{11}-I_1)\<\ker
A_{01}$. Here $I_1$ denotes the identity operator on $X_1$. So we have
\[\ker(A_{11}-I_1)\<\ker(A_{11}-I_1)\cap\ker A_{01}\cap X_1\ =\ \ker(A-I)\cap
X_1\ =\ \{0\}.\]

Now we distinguish two cases. If $A$ is $h$-unitary, let $y\in \dom (A)\cap X_1$\, and
$x\in\ker(A-I)$. Then
\[
h(x,Ay)\ =\ h(Ax,Ay)\ =\ h(x,y)\ =\ 0
\]
by \eqref{e:h-orthogonal-domain}. So, the range $\ran(A|_{\dom (A)\cap X_1})$ is $h$-orthogonal to
$\ker(A-I)$ and, hence, contained in $X_1$\,. Hence $A_{01}=0$. We observe that $A-I$ is closed as
bounded perturbation of the closed operator $A$; it follows that the component $A_{11}$ and the
operator $A_{11}-I_1$ are closed in $X_1$\,. That proves that $A_{11}-I_1$ has a bounded inverse.

If, on the other side, $A$ is bounded, then both $A_{01}$ and $A_{11}$ are bounded and we have
\begin{eqnarray*}
\index(A_{11}-I_1)&\ =\ &\index \bigl(\diag(I_0,A_{11}-I_1)\bigr)\\
&\ =\ &\index\bigl((A-I)\diag(I_0,I_1)\bigr)\\
&\ =\ &\index(A-I)+\index\bigl(\diag(I_0,I_1)\bigr)\  =\ 0.
\end{eqnarray*}
By $\ker(A_{11}-I_1)=\{0\}$ we have $A_{11}-I_1$ surjective. By the Closed Graph Theorem, it
follows that $(A_{11}-I_1)\ii$ is bounded and so $A_{11}-I_1$ has a bounded inverse.

We conclude that in both cases $A_1$ has no spectrum near 1. From the decomposition
\eqref{e:a-block-form} we get $\s(A)=\s(I_0)\cup\s(A_1)$ with $\s(I_0)=\{1\}$. So, if $1\in\s(A)$
it is an isolated point of $\s(A)$ of multiplicity $\dim\ker(A-I)$.

\noi {\it b}. Since $A$ is admissible with respect to ${\mathfrak m}$, there exists a bounded open
subset $N$ of $\C$ such that $\sigma(A)\cap {\mathfrak m}=\sigma(A)\cap N$ and $\dim\image
P_N(A)<+\infty$. Then $P_N(A)AP_N(A)$ defined on the finite dimensional vector space $\image
P_N(A)$ is unitary with respect to $h|_{\image P_N(A)}$. Thus we have $\sigma(P_N(A)AP_N(A))\subset
S^1$ and%
\[
\sigma(A)\cap {\mathfrak m}=\sigma(A)\cap N=\sigma(P_N(A)AP_N(A))\subset S^1\cap {\mathfrak
m}\subset S^1\cap\ell=\{1\}.
\]

\noi {\it c}. From our assumption it follows that $\s(A)\cap \partial N = \emptyset$ and, actually,
$\s(A_s)\cap
\partial N = \emptyset$ for $s$ sufficiently small. Then
\[
\left\{P_N(A_s) \ :=\  -\frac 1{2\pi i}\int_{\partial N} (A-\la I)\ii d\la\right\}
\]
is a continuous family of projections. From \cite[Lemma I.4.10]{Ka76} we obtain
\[
\dim\ran P_N(A_s) \ =\ \dim\ran P_N(A)\ <\ +\infty \quad \tand\quad
P_N(A_s)A_s\ \<\ A_sP_N(A_s),
\]
and from \cite[Lemma III.6.17]{Ka76} we get $ \s(A_s)\cap
N=\s(P_N(A_s) A_s P_N(A_s)).$ Since all operators $P_N(A_s) A_s
P_N(A_s)$ are unitary with respect to $h_s|_{\ran P_N(A_s)}$, it
follows $\s(P_N(A_s) A_s P_N(A_s))$$\< S^1$\,.

\smallskip
\noi {\it d}. By (c) and Lemma \ref{l:16} of Appendix \ref{ss:A3}, for any $t\in[0,1]$, there
exists $\e(t)\in(0,1)$ and $\delta(t)>0$ such that $\{A_s\}$,
$s\in(t-\delta(t),t+\delta(t))\cap[0,1]$ is a family of admissible operators that is spectral
continuous near ${\mathfrak m}_{\e(t)}$. The open cover $\{(t-\delta(t),t+\delta(t))\}$,
$t\in[0,1]$ of $[0,1]$ has a finite subcover $\{(t_k-\delta(t_k),t_k+\delta(t_k))\}$,
$k=1,\ldots,n$. Set $\e=\min\{\e(t_k);k=1,\ldots,n\}$. Then $\{A_s\}$,
$s\in(t_k-\delta(t_k),t_k+\delta(t_k))\cap[0,1]$ is a family of admissible operators that is
spectral continuous near ${\mathfrak m}_{\e}$, and $\{A_s\}$, $s\in [0,1]$ is a family of
admissible operators that is spectral continuous near ${\mathfrak m}_{\e}$.

\noi {\it e}. By (d), such ${\mathfrak m}$ do exist. By (b), $\SF_{\mathfrak m}\{A_s\}$ does not
depend on the choice of ${\mathfrak m}$. Thus our concept of the spectral flow relative
$\ell=(0,+\infty)$ is well-defined.

\end{proof}

Thus, it follows that any $h$-unitary operator $A$ with $A-I$ Fredholm of index $0$ has the same
spectral properties near $|\la|=1$ as unitary operators in Hilbert space with the additional
property that 1 is an isolated point of the spectrum of finite multiplicity.

This now permits us to define the Maslov index in weak symplectic analysis.


\addtocontents{toc}{\medskip\noi}

\section{Maslov index in weak symplectic analysis}\label{s:maslov}

Now we turn to the geometry of curves of Fredholm pairs of Lagrangian subspaces in weak symplectic
Banach spaces. We show how the usual definition of the Maslov index can be suitably extended and
derive basic and more intricate properties.

\subsection{Definition and basic properties of the Maslov index}
\label{ss:maslov} Our data for defining the Maslov index are a {\em
continuous} family
$
\{(X,\w_s, X_s^+,X_s^-)\}
$
 of weak symplectic Banach spaces with {\it continuous}
splitting and a {\em continuous} family $\{(\la_s,\mu_s)\}$ of Fredholm pairs of Lagrangian
subspaces of $\{(X,\w_s)\}$ of index 0. Our first task is defining the involved ``continuity".

\begin{definition}\label{d:continuous-banach}
Let $X$ be a fixed complex Banach space and $\{\w_s\}$ a family of
weak symplectic forms for $X$. Let $(X,\w_s,X_s^+, X_s^-)$ be a
family of symplectic splittings of $(X,\w_s)$ in the sense of
Definition \ref{d:symplectic-splitting}.

\noi (a) The family $\{(X,\w_s,X_s^+,X_s^-)\}$ will be called {\em continuous} if the family of
forms $\{\omega_s\}$ is continuous, and the families $\{X_s^\pm\}$ are continuous as closed
subspaces of $X$ in the gap topology. Equivalently, we may demand that the family $\{P_s\}$ of
projections
\[
P_s\colon  x+y \ \mapsto\  x, \quad \text{ for $x\in X_s^+ \tand y\in X_s^-$}\,,
\]
is continuous.


\noi (b)  Let $\{(X,\w_s, X_s^+,X_s^-)\}$, $s\in[a,b]$  be a continuous family of symplectic
splittings with induced inner products $h_s^{\pm}=\mp\omega|_{X^{\pm}}$. Let $\{(\la_s,\mu_s)\}$ be
a continuous curve of Fredholm pairs of Lagrangian subspaces of index 0. Let $U_s\colon \dom
(U_s)\to X^-_s$, resp. $V_s\colon \dom(V_s)\to X^-_s$ be closed $(h_s^+,h_s^-)$-unitary operators
with $\GGG(U_s)=\la_s$ and $\GGG(V_s)=\mu_s$\/. We define the {\em Maslov index} of the curve
$\{\la_s,\mu_s\}$ with respect to $P_s$ by
\begin{equation}\label{e:maslov}
\Mas\{\la_s,\mu_s;P_s\} \ :=\  \SF_{\ell}\Bigl\{\begin{pmatrix} 0&U_s\\
V_s\ii&0\end{pmatrix}\Bigr\} \/,
\end{equation}
where $V\ii$ denotes the algebraic inverse of the closed injective operator $V$ and
$\ell:=(0,+\infty)$ and with upward co-orientation. The spectral flow $\SF_{\ell}$ is defined in
the sense of Proposition \ref{p:unitary-admissible}e.
\end{definition}

\begin{remark}\label{r:full-domain}
Let $\{(X,\w_s,X_s^+,X_s^-)\}$ be a continuous family. A curve $\{\la_s\}$ of Lagrangian subspaces
is continuous (i.e., $\{\la_s=\GGG(U_s)\}$ is continuous as a curve of closed subspaces of $X$), if
and only if the family $\{S_{s,s_0}\circ U_s\circ S_{s,s_0}\ii\}$ is continuous as a family of
closed, generally unbounded operators in the space $\ran P_{s_0}$. Here $U_s$ denotes the
generating operator $U_s\colon \dom U_s\to X_s^-$ with $\mathfrak{G}(U_s)=\la_s$ (see Lemma
\ref{l:lagrangian-representation}); $s_0\in [0,1]$ is chosen arbitrarily to fix the domain of the
family; and
\[
S_{s,s_0}\colon \ran P_s\too \ran P_{s_0}
\]
is a bounded operator with bounded inverse which is defined in the
following way (see also \cite[Section I.4.6, pp. 33-34]{Ka76}):
\[
S_{s,s_0}\ :=\  S_{s,s_0}' (I-R)^{-1/2} \ =\  (I-R)^{-1/2} S_{s,s_0}'\,,
\]
where
\[
R \ :=\ (P_s-P_{s_0})^2 \qquad \tand\qquad S_{s,s_0}' \ :=\ P_{s_0}P_s+(I-P_{s_0})(I-P_s).
 \]

\end{remark}

\medskip

The main result of our paper is

\begin{theorem}\label{t:main} The Maslov index of Definition \ref{d:continuous-banach}b is well-defined.
\end{theorem}

\begin{proof} By  a series of lemmas below, we check that the family of block matrices on the right side of \eqref{e:maslov} satisfies
the condition of Proposition \ref{p:unitary-admissible}e. Then our theorem follows. Note that we do
not need the continuity of $\omega_s$ for our chain of arguments.
\end{proof}

\begin{lemma}\label{l:boxplus}
Let $(X,\w)$ be a weak symplectic Banach space. Let $\D$ denote the diagonal (i.e., the canonical
Lagrangian) in the product symplectic space $X\boxplus X:=(X,\w)\oplus (X,-\w)$, and $\la,\mu$
Lagrangian subspaces of $(X,\w)$. Then
\[
(\la,\mu)\in\Ff\Ll^2(X)\quad \iff\quad
(\la\boxplus\mu,\D)\in\Ff\Ll^2\bigl(X\boxplus X\bigr)
\]
and
\[
\index(\la,\mu)\ =\ \index(\la\boxplus\mu,\D),
\]
where $\la\boxplus\mu:=\{(x,y)\mid x\in \la,y\in\mu\}$.
\end{lemma}

\begin{proof}
Clearly $\lambda\boxplus\mu$ and $\Delta$ are Lagrangian subspaces of $X\boxplus X$. Since%
\[
(\lambda\boxplus\mu)\cap\Delta\ =\ \{(x,x)|x\in\lambda\cap\mu\}\ \simeq\ \lambda\cap\mu,%
\]
these spaces have the same dimension. Set%
\[
\Delta^{\prime}\ :=\ \{(x,-x)|x\in X\}\quad \tand\quad
\Delta^{\prime}_{\lambda+\mu}\ :=\ \{(x,-x)|x\in \lambda+\mu\}.%
\]
Then we have%
\begin{align*}
\lambda\boxplus\mu+\Delta\ &=\ \{(x,y)+(\xi,\xi)|x\in\lambda,y\in\mu,\xi\in X\}\\
\ &=\
\{(\frac{x-y}{2},\frac{y-x}{2})+(\frac{x+y}{2}+\xi,\frac{x+y}{2}+\xi)|x\in\lambda,y\in\mu,\xi\in
X\}\ =\ \Delta+\Delta^{\prime}_{\lambda+\mu}.%
\end{align*}
So the following holds:
\[
\frac{X\boxplus X}{\lambda\boxplus\mu+\Delta}\ =\
\frac{\Delta\oplus\Delta^{\prime}}{\Delta\oplus\Delta^{\prime}_{\lambda+\mu}}
\ \simeq\ \frac{\Delta^{\prime}}{\Delta^{\prime}_{\lambda+\mu}}\ \simeq\ \frac{X}{\lambda+\mu},%
\]
and they have the same dimension. Lemma \ref{l:boxplus} is proved.
\end{proof}

\begin{lemma}\label{l:2} Let $(X,\omega,X^+,X^-)$ be a weak symplectic Banach space with symplectic splitting.
Set $h^{\pm}:=\mp i\omega|_{X^{\pm}}$. Let $\Delta$ denote the diagonal in the symplectic space
$X\boxplus X$. Let $\lambda$ and $\mu$ be two Lagrangian subspaces of $X$ with generator $U$, $V$
respectively. Then we have:

\noi (a) The pair $(\lambda,\mu)$ is Fredholm of index $0$ if and only if $\begin{pmatrix} 0&U\\
V\ii&0\end{pmatrix}-I$ is of index $0$.

\noi (b) The matrix $\begin{pmatrix} 0&U\\
V\ii&0\end{pmatrix}$ is $(h^-\oplus h^+)$-unitary.
\end{lemma}

\begin{proof} \textit{a}. Let $P\colon X\to X^+$ denote the projection corresponding to the splitting.
Let $(\la,\mu)\in\Ff\Ll^2(X)$. Then we have
\[
\begin{pmatrix} 0&U\\ V\ii&0\end{pmatrix}
\ =\ \begin{pmatrix} U&0\\0& V\ii\end{pmatrix}
\begin{pmatrix} 0&I_{X^+}\\ I_{X^-}& 0 \end{pmatrix},
\]
and
\begin{equation}\label{e:product-graph}
\wt\GGG\begin{pmatrix}U&0\\0&V\ii\end{pmatrix}\ =\ \la\boxplus\mu,\quad\tand\quad
\wt\GGG\begin{pmatrix}0&I_{X^-}\\I_{X^+}&0\end{pmatrix}\ =\ \D,
\end{equation}
where $\wt\GGG$ denotes the graph of closed operators from $\range \Pp$ to $\range (I-\Pp)$ with
$\Pp:=P\boxplus (I-P)$. Consequently, $\begin{pmatrix} 0&U\\
V\ii&0\end{pmatrix}-I$ is a Fredholm operator of index $0$. Conversely, by Equation
\eqref{e:product-graph}, we can derive that $(\la,\mu)$ is a Fredholm pair of index 0 if $\begin{pmatrix} 0&U\\
V\ii&0\end{pmatrix}-I$  is a Fredholm operator of index 0.

\noi \textit{b}. For all $x_1, x_2\in X^+$ and $y_1,y_2\in X^-$, we have
\begin{align*}
(h^-\oplus h^+)&\Bigl(\begin{pmatrix} 0&U\\
V\ii&0\end{pmatrix}\begin{pmatrix} y_1\\ x_1\end{pmatrix}\/,\/\begin{pmatrix} 0&U\\
V\ii&0\end{pmatrix}\begin{pmatrix} y_2\\ x_2\end{pmatrix}\Bigr) \\
&=\ (h^-\oplus h^+)\left(\begin{pmatrix} Ux_1\\ V^{-1}y_1\end{pmatrix}\/,\/\begin{pmatrix} Ux_2\\ V^{-1}y_2\end{pmatrix}\right)\\
&=\ h^-(Ux_1,Ux_2)+h^+(V^{-1}y_1,V^{-1}y_2)\\
&=\ h^+(x_1,x_2)+h^-(y_1,y_2)\ =\ (h^-\oplus h^+)\left(\begin{pmatrix} y_1\\
x_1\end{pmatrix}\/,\/\begin{pmatrix} y_2\\ x_2\end{pmatrix}\right).
\end{align*}
\end{proof}

\begin{lemma}\label{l:3} The family $\left\{\begin{pmatrix} 0&U_s\\
V_s\ii&0\end{pmatrix}\right\}$ in \eqref{e:maslov} is a continuous family of closed operators such
that
$\begin{pmatrix} 0&U_s\\
V_s\ii&0\end{pmatrix}-I$ is Fredholm of index $0$ for each $s\in[a,b]$.
\end{lemma}

\begin{proof} The proof of the above Lemma \ref{l:2}a shows that the graph of $\left\{\begin{pmatrix} 0&U_s\\
V_s\ii&0\end{pmatrix}\right\}$ is a continuous family of closed subspaces of $X\oplus X$. Then the
family
$\left\{\begin{pmatrix} 0&U_s\\
V_s\ii&0\end{pmatrix}\right\}$ is a continuous family of closed operators. By Lemma \ref{l:2}a we
have that
$\begin{pmatrix} 0&U_s\\
V_s\ii&0\end{pmatrix}-I$ is Fredholm of index $0$ for each $s\in[a,b]$.
\end{proof}
That ends the proof of Theorem \ref{t:main}.

\medskip
The proof of Lemma \ref{l:2}a leads to the following important result:

\begin{proposition}\label{p:boxplus}
Let $\{X,\w_s\}$, $s\in[a,b]$ be a continuous family of weak symplectic forms for $X$ with a
continuous family of symplectic splittings $(X,\w_s,X_s^+, X_s^-)$ in the sense of Definition
\ref{d:continuous-banach}a and a corresponding family of projections $\{P_s\colon X\to X_s^+\}$.
Let $\{\la_s,\mu_s\}$, $s\in[a,b]$ be a continuous curve in $\Ff\Ll^2(X)$. We denote the generating
operators by $U_s$, respectively $V_s$\/.


\noi (a) If $V_s$ is bounded and has a bounded inverse for each $s\in [0,1]$, then we have
\begin{equation}\label{e:maslov2}
\Mas\{\la_s,\mu_s; P_s\}\ =\ \SF_{\ell}\{U_sV_s\ii\},
\end{equation}
where $\ell:=(0,+\infty)$ with upward co-orientation. The spectral flow is defined in the sense of
Proposition \ref{p:unitary-admissible}e.

\noi (b) We have
\begin{align}\label{e:maslov1}
\Mas\{\la_s\boxplus\mu_s,\D;{\Pp}_s\}
\ &=\ \Mas\{\la_s,\mu_s;P_s\}\\
& =\ \Mas\{\mu_s,\la_s;I-P_s\}, \qquad\quad\text{ in $(X,-\w_s)$},
\label{e:maslov3}\\
& =\ \Mas\{\D,\la_s\boxplus\mu_s;I-{\Pp}_s\}, \quad\!\text{in $(X,-\w_s)\boxplus (X,\w_s)$},
\label{e:maslov4}
\end{align}
where $\Pp_s:=P_s\boxplus (I-P_s)$\,.
\end{proposition}

\begin{proof}
\textit{a}. By our assumption, we have
$$\dim\ker(z^2I-U_sV_s^{-1})=\dim\ker\left(zI-\begin{pmatrix} 0&U_s\\
V_s\ii&0\end{pmatrix}\right)$$ for all $z\in\C\setminus\{0\}$. By Proposition
\ref{p:unitary-admissible} and Lemma \ref{l:2}, both the total algebraic multiplicities of the
spectrum of the matrices $U_sV_s^{-1}$ and $\begin{pmatrix} 0&U_s\\
V_s\ii&0\end{pmatrix}$ near $1$ are $\dim\ker(I-U_sV_s^{-1})$, and each spectrum of them with
finite algebraic multiplicity is on $S^1$. Then (a) follows from the definition of the Maslov index
and Proposition \ref{p:unitary-admissible}.

\noi\textit{b}. Let $\wt\GGG$ denote the graph of closed operators from $\range \Pp_s$ to $\range
(I-\Pp_s)$\,. By Equations \eqref{e:maslov}, \eqref{e:product-graph} and \eqref{e:maslov2} we have
\begin{align*}
\Mas\{\la_s\boxplus\mu_s,\D;{\Pp}_s\} &\ =\  \Mas
\Bigl\{\wt\GGG\begin{pmatrix}U_s&0\\0&V_s\ii\end{pmatrix},
\wt\GGG\begin{pmatrix}0&I_{X_s^-}\\I_{X_s^+}&0\end{pmatrix};
\Pp_s\Bigr\} \\
&\ =\  \SF_{\ell}\Bigl\{\begin{pmatrix} U_s&0\\0& V_s\ii\end{pmatrix}
\begin{pmatrix} 0&I_{X_s^+}\\ I_{X_s^-}& 0 \end{pmatrix}\Bigr\}\\
&\ =\  \SF_{\ell}\Bigl\{\begin{pmatrix} 0&U_s\\
V_s\ii&0\end{pmatrix}\Bigr\} \ =\ \Mas\{\la_s,\mu_s; P_s\}.
\end{align*}
So (\ref{e:maslov1}) is proved.

For the symplectic space $(X,-\omega_s)$ with symplectic splitting $(X,-\omega_s,X_s^-,X_s^+)$, the
generating operators of $\lambda_s,\mu_s$ are $U_s^{-1},V_s^{-1}$ respectively. Note that
\[
\begin{pmatrix}0&V_s^{-1}\\U_s&0\end{pmatrix}
=\begin{pmatrix}0&I_{X^+_s}\\I_{X^-_s}&0\end{pmatrix}
\begin{pmatrix} 0&U_s\\ V_s\ii&0\end{pmatrix} \begin{pmatrix}0&I_{X^+_s}\\I_{X^-_s}&0\end{pmatrix}.%
\]
By the definition of the Maslov index we have
$$\Mas\{\mu_s,\lambda_s;I-P_s\}_{\text{in }(X,-\omega)}
\ =\ \SF_{\ell}\left\{\begin{pmatrix}0&V_s^{-1}\\U_s&0\end{pmatrix}\right\}\ =\
\SF_{\ell}\left\{\begin{pmatrix} U_s&0\\0& V_s\ii\end{pmatrix}\right\}\ =\
\Mas\{\lambda_s,\mu_s;P_s\}.$$ So (\ref{e:maslov3}) is proved. (\ref{e:maslov4}) follows from
(\ref{e:maslov3}) and (\ref{e:maslov1}).
\end{proof}

From the properties of our general spectral flow, as observed at the
end of our Appendix, we get all the basic properties of the Maslov
index (see S. E. Cappell, R. Lee, and E. Y. Miller \cite[Section
1]{CaLeMi94} for a more comprehensive list).

\begin{proposition}\label{p:maslov-properties}
(a) The Maslov index is {\em invariant under homotopies} of curves
of Fredholm pairs of Lagrangian subspaces with fixed endpoints. In
particular, the Maslov index is invariant under {\em
re-parametrization} of paths.

\noi (b) The Maslov index is additive under {\em catenation}, i.e.,
\[
\Mas\bigl\{{\la}_1* {\la}_2,
\mu_1 * \mu_2;P_s*Q_s\bigr\} \ =\ \Mas\bigl\{{\la}_1,\mu_1;P_s\bigr\} +
\Mas\bigl\{{\la}_2,\mu_2;Q_s\bigr\}\,,
\]
where $\{{\la}_i(s)\},\{{\mu}_i(s)\},\, i=1,2$ are continuous paths
with ${\la}_1(1)={\la}_2(0)$, ${\mu}_1(1)={\mu}_2(0)$ and
\[
({\la}_1*{\la}_2)(s) \ :=\  \begin{cases}
{\la}_1(2s), \quad & 0\leq s\leq \12\,,\\
{\la}_2(2s-1), \quad & \12 < s\leq 1\,,
\end{cases}
\]
and similarly $\mu_1*\mu_2$ and $\{P_s\}*\{Q_s\}$\,.

\noi (c) The Maslov index is {\em natural} under symplectic action:
let $\{(X',\w_s')\}$ be a second family of symplectic Banach spaces
and let
\begin{equation*}
L_s\in \operatorname{Sp}(X,\w_s;X',\w_s')\ \ :=\ \ \{L\in \Bb(X,X')\mid L \text{ invertible and
$\w_s'(Lx,Ly)=\w_s(x,y)$}\},
\end{equation*}
such that $\{L_s\}$ is a continuous family of bounded operators. Then $\{X'=L_s(X_s^+)\oplus
L_s(X_s^-)\}$ is a continuous family of symplectic splittings of $\{(X',\w_s')\}$ inducing
projections $\{Q_s\}$, and we have
\[
\Mas\{\la_s,\mu_s;P_s\}\ =\  \Mas\{L_s\la_s,L_s\mu_s;Q_s\}.
\]

\noi (d) The Maslov index {\em vanishes}, if $\dim (\la_s \cap \mu_s)$ is constant for all $s\in
[0,1]$.

\noi (e) {\em Flipping}. We have
\[
\Mas\{\la_s,\mu_s;P_s\}+ \Mas\{\mu_s,\la_s;P_s\} \ =\  \dim
\la_0\cap\mu_0 -\dim\la_1\cap\mu_1\/.
\]
\end{proposition}


We can not claim that the Maslov index, $\Mas\{\la_s,\mu_s;P_s\}$ is always independent of the
splitting projection $P_s$ in general Banach spaces. However, we have the following result.

\begin{proposition}\label{p:mas-independence}
Let $\{(X,\w_s)\}$ be a continuous family of {\em strong} symplectic Banach spaces and let
$\{X=X_{s,{\rho}}^+\oplus X_{s,{\rho}}^-\}$ be two continuous families of symplectic splittings in
the sense of Definition \ref{d:continuous-banach}a with projections $P_{s,{\rho}}\colon X\to
X_{s,{\rho}}^+$ for $s\in [0,1]$ and ${\rho}=0,1$. Let $\{(\la_s,\mu_s)\}$ be a continuous curve of
Fredholm pairs of Lagrangian subspaces of $\{(X,\w_s)\}$. Then

\noi (a) \quad $\index (\la_s,\mu_s)=0$ for all $s\in [0,1]$; and

\noi (b) \quad $\Mas\{\la_s,\mu_s;P_{s,0}\} =
\Mas\{\la_s,\mu_s;P_{s,1}\}$ .
\end{proposition}

\begin{note} Commonly, one assumes $J^2=-I$ in strong symplectic analysis
and defines the Maslov index with respect to the induced
decomposition. In view of Lemma \ref{l:weak-strong}, the point of
the preceding proposition is that the Maslov index is independent of
the choice of the metrics.
\end{note}

\begin{proof} {\it a}. Using $-i\w_s$, we make $(X,\w_s)$ into a symplectic Hilbert
space and deform the metric such that $J_s^2=-I$. Clearly, the
dimensions entering into the definition of the Fredholm index do not
change under the deformation. So, we are in the well-studied
standard case.

\noi {\it b}. We recall that our two families of symplectic splittings define two families of
Hilbert structures for $X$ defined by
\[
\lla x_{s,{\rho}}^+ + x_{s,{\rho}}^- , y_{s,{\rho}}^+ + y_{s,{\rho}}^- \rra_{s,{\rho}} \
:=\  -i\w_s(x_{s,{\rho}}^+,y_{s,{\rho}}^+) + i\w_s(x_{s,{\rho}}^-,y_{s,{\rho}}^-)
\]
for $x_{s,{\rho}}^+, y_{s,{\rho}}^+ \in H_{s,{\rho}}^+, x_{s,{\rho}}^-, y_{s,{\rho}}^- \in
H_{s,{\rho}}^-$, and ${\rho}=0,1$. For any ${\rho}\in [0,1]$ we define
\[
\lla x,y\rra_{s,{\rho}} \ :=\  (1-{\rho})\lla x,y\rra_{s,0} + {\rho}\lla x,y\rra_{s,1}\/.
\]
Then all $(X,\lla\cdot,\cdot\rra_{s,{\rho}})$ are Hilbert spaces.

Define $J_{s,{\rho}}$ by $\w_s(x,y)=\lla J_{s,{\rho}}x,y\rra_{s,{\rho}}$ and let $X_{s,{\rho}}^\pm$
denote the positive (negative) space of $-iJ_{s,{\rho}}$ and $P_{s,{\rho}}$ the orthogonal
projection of $X$ onto $X_{s,{\rho}}^+$\/. Then the two-parameter family $\{J_{s,{\rho}}\}$ is a
continuous family of invertible operators, $\{P_{s,{\rho}}\}$ is continuous, and
$\{H_{s,{\rho}}^+\}$ is continuous. So $\Mas\{\lambda_s,\mu_s;P_{s,\rho}\}$ is well-defined. By
Proposition \ref{p:maslov-properties}a and \ref{p:maslov-properties}b, we have
\begin{multline*}\Mas\{\lambda_0,\mu_0;P_{0,\rho};0\le\rho\le 1\}+\Mas\{\lambda_s,\mu_s;P_{s,0};0\le s\le 1\}
\\
=\ \Mas\{\lambda_s,\mu_s;P_{s,1};0\le s\le 1\}+\Mas\{\lambda_1,\mu_1;P_{1,\rho};0\le\rho\le 1\}.
\end{multline*}
By Proposition \ref{p:maslov-properties}d, we have
$$\Mas\{\lambda_0,\mu_0;P_{0,\rho};0\le\rho\le 1\}\ =\ \Mas\{\lambda_1,\mu_1;P_{1,\rho};0\le\rho\le 1\}\ =\ 0.$$
So we obtain
$$\Mas\{\lambda_s,\mu_s;P_{s,0};0\le s\le 1\}\ =\ \Mas\{\lambda_s,\mu_s;P_{s,1};0\le s\le 1\}.$$
\end{proof}

\subsection{Comparison with the real (and strong) category}\label{ss:comparison2real}

For a fixed strong symplectic Hilbert space $X$, choosing one single Lagrangian subspace $\la$
yields a decomposition $X=\la\oplus J\la$\,. This decomposition was used in \cite[Definition
1.5]{BoFu98} (see also \cite[Theorem 3.1]{BoFuOt01} and \cite[Proposition 2.14]{FuOt02}) to give
the first functional analytic definition of the Maslov index, though under the somewhat restrictive
(and notationally quite demanding) assumption of {\it real} symplectic structure. Up to the sign,
our Definition \ref{d:continuous-banach}b is a true generalization of that previous definition.
More precisely:

Let $(H,\w)$ be a real symplectic Hilbert space with
\[
\w(x,y)\ =\ \lla Jx,y\rra,\ J^2\ =\ -I, \ J^t\ =\ -J.
\]
Clearly, we obtain a symplectic decomposition $H^+\oplus H^- =
H\otimes\C$ with the induced complex strong symplectic form
$\w_{\C}$ by
\[
H^\pm\ :=\  \{(I\mp i J)\z\mid \z\in H\}.
\]

\begin{definition}\label{d:complex-generator}
We fix one (real) Lagrangian subspace $\la\< H$.
\newline a) Then there is a real linear
isomorphism $\f\colon H\cong\la\otimes\C$ defined by $\f(x+Jy):=x+i y$ for all $x,y\in\la$.
\newline b) For $A=X+JY\colon H\to H$ with $X,Y\colon H\to H$ real linear and
\begin{equation}\label{e:real-complex-def0}
X(\la)\< \la,\ Y(\la)\< \la,\quad\tand\quad XJ\ =\ JX,\ YJ\ =\ JY,
\end{equation}
we define
\[
\f_*(A)\ :=\ \f\circ A\circ \f\ii \ =\ X+i Y, \quad\ol A_{\la}\ :=\ X-JY,\quad A^{t_{\la}}\ :=\
X^t+JY^t,\,
\]
where $X^t,Y^t$ denote the real transposed operators.
\newline c) Let $\mu$ be a second Lagrangian subspace of $H$ and let $\wt V:H\to H$ with $\wt
VJ=J\wt V$ be a real generating operator for $\mu$ with respect to the orthogonal splitting
$H=\la\oplus J\la$, i.e., $\mu=\wt V(J\la)$ and $\f_*(\wt V)$ is unitary. Then we define the
complex generating operator for $\mu\otimes \C$ with respect to $\la$ by $S_{\la}(\wt V):=\f_*(\wt
V)\f_*\bigl(\wt V^{t_{\la}}\bigr)$.
\end{definition}

\begin{note} The complex generating operator for $\mu\otimes \C$ with respect to $\la$ was
defined by J. Leray in \cite[Section I.2.2, Lemma 2.1]{Le78} and elaborated in the references given
at the beginning of this subsection.
\end{note}

\begin{lemma}\label{l:mas-bofu=mas-zhu}
Let $(\la,\mu)$ be any pair of Lagrangian subspaces of $H$ (in the real category). Let $\wt V\colon
H\to H$ with $\wt VJ=J\wt V$ be a real generating operator for $\mu$ with respect to the orthogonal
splitting $H=\la\oplus J\la$. Let $U,V\colon H^+\to H^-$ denote the unitary generating operators
for $\la\otimes\C$ and $\mu\otimes\C$, i.e., we have
\[
\la\otimes \C\ =\ \GGG(U)\ \tand \ \mu\otimes \C\ =\ \GGG(V).
\]
Then we have $VU\ii=-\ol{S_{\la}(\wt V)}$, where $S_{\la}(\wt V)$ denotes the complex generating
operator for $\mu\otimes \C$ with respect to $\la$, as introduced in the preceding Definition.
\end{lemma}

\begin{proof}
At first we give some notations used later. For $\z=x+Jy\in H$ with $x,y\in\la$, we define $
\ol{\z}_{\la}:=\f\ii\bigl(\ol{\f(\z)}\bigr) =x-Jy$. Moreover, for $A=X+JY\colon H\to H$ with
$X,Y\colon H\to H$ real linear with (\ref{e:real-complex-def0}), we define $\wt
S_{\la}(A):=AA^{t_{\la}}$. Then we have $S_{\la}(A)=\f_*\bigl(\wt
S_{\la}(A)\bigr)$.

Now we give explicit descriptions of $U$ and $V$. It is immediate
that $U$ takes the form
\[
\begin{matrix}
U&\colon &H^+&\too&H^-\\
\ &\ &(I-i J)\z&\mapsto&(I+i J)\ol\z_{\la}\ .
\end{matrix}
\]
By the definition of $\wt V$, we have
\[
\mu\ =\ \wt V(J\la)\ =\ \{2\wt VJx+2i \wt VJy\mid x,y\in\la\}.
\]
We shall find $V\colon (I-i J)\z\mapsto (I+i J)\z_1$ with $\z,\z_1\in H$ such that
$\GGG(V)=\mu\otimes\C$, i.e., we shall find $\z_1$ to $\z=x+Jy$ such that
\begin{equation}\label{e:real-complex}
(I-i J)\z + (I+i J)\z_1 \ =\  2\wt VJx+2i \wt VJy \quad\text{ for
all $x,y\in \la$}.
\end{equation}
Comparing real and imaginary part of \eqref{e:real-complex} yields
$\z+\z_1=2\wt VJx$ and $-i J(\z-\z_1)=-i\wt VJy$, so
\[
\z\ =\ \wt V(Jx-y) \quad\tand\quad \z_1\ =\ \wt V(Jx+y).
\]
From the left equation we obtain $\ol\z_{\la}=-\ol{\wt
V}_{\la}(Jx+y)$. Since $\f_*(\wt V)$ is unitary, we obtain from the
right side
\[
\z_1\ =\ \wt V(Jx+y)\ =\ -\wt V\ol{\wt V}_{\la}\ii \ol\z_{\la} \ =\
-\wt V\wt V^t\ol\z_{\la}\ =\  -\wt S_{\la}(\wt V)\ol\z_{\la}\/.
\]
This gives
\[
\begin{matrix}
V&\colon &H^+&\too&H^-\\
\ &\ &(I-i J)\z&\mapsto&-(I+i J)\wt S_{\la}(\wt V)\ol\z_{\la}\/.
\end{matrix}
\]
So for all $z_1:=(I+i J)\z_1$ with $\z_1\in H$, we have
\begin{eqnarray*}
VU\ii z_1&\ =\ &-(I+i J)\wt S_{\la}(\wt V)\z_1\\
&\ =\ &-\wt S_{\la}(\wt V)(I+i J)\z_1\\
&\ =\ &-\wt S_{\la}(\wt V)(I-i J)\ol{\f(\z)}\\
&\ =\ &-\ol{\f_*(\wt S_{\la}(\wt V))}(I-i J)\ol{\f(\z)}\\
&\ =\ &-\ol{S_{\la}(\wt V)}(I+i J)\z_1\\
&\ =\ &-\ol{S_{\la}(\wt V)}z_1.
\end{eqnarray*}
That is, $VU\ii=-\ol{S_{\la}(\wt V)}$.
\end{proof}

With the preceding notation, we recall from \cite[Definition
1.5]{BoFu98} the definition of the Maslov index
\begin{equation}\label{e:maslov-bf}
\Mas_{\operatorname{BF}}\{\mu_s,\la\}\ :=\  \SF_{\ell'}\{S_{\la}(\wt V_s)\}
\end{equation}
of a continuous curve $\{\mu_s\}$ of Lagrangian subspaces in real symplectic Hilbert space $H$
which build Fredholm pairs with one fixed Lagrangian subspace $\la$. Here $\ell':=(-1-\e,-1+\e)$
with downward orientation.

\begin{corollary}\label{c:mas-bofu=mas-zhu}
\[
\Mas\{\la\otimes\C,\m_s\otimes \C\}\ =\
-\Mas_{\operatorname{BF}}\{\mu_s,\la\}.
\]
\end{corollary}

\begin{proof}
Let $\ell,\ell'$ denote small intervals on the real line close to 1, respectively -1 and give
$\ell$ the co-orientation from $-i$ to $+i$ and $\ell'$ vice versa. We denote by $\ell^-$ the
interval $\ell$ with reversed co-orientation. Then by our definition in \eqref{e:maslov},
Proposition \ref{p:boxplus}a, elementary transformations, the preceding lemma, and the definition
recalled in \eqref{e:maslov-bf}:
\begin{align*}
\Mas\{\la\otimes\C,\m_s\otimes \C\}
&\ =\  -\SF_{\ell}\{UV_s\ii\}\ =\  -\SF_{\ell^-}\{V_sU\ii\} \\
&\ =\  -\SF_{\ell^-}\{-\ol{S_{\la}(\wt V_s)}\}
\ =\  -\SF_{\ell}\{-S_{\la}(\wt V_s)\}\\
& \ =\  -\SF_{\ell'}\{S_{\la}(\wt V_s)\} \ =\
-\Mas_{\operatorname{BF}}\{\mu_s,\la\}.
\end{align*}
\end{proof}

\subsection{Invariance of the Maslov index under embedding}\label{ss:embedding}

We close this section by discussing the invariance of the Maslov index under embedding in a larger
symplectic space, assuming a simple regularity condition.

\begin{lemma}\label{l:maslov-embedding}
Let $\{(X,\w_s,X^+_s,X^-_s)\}$ be a continuous family of symplectic splittings for a (complex)
Banach space $X$ and let $\{(\la_s,\mu_s)\in \Ff\Ll^2(X,\w_s)\}$ be a continuous curve with
$\index(\la_s,\mu_s)=0$ for all $s\in [0,1]$. Let $Y$ be a second Banach space with a linear
embedding $Y\into X$ (in general neither continuous nor dense). We assume that
\[
\wt{\w}_s\ :=\ \w_s|_{Y\times Y} \quad\tand\quad Y_s^\pm\ :=\ X_s^\pm\cap Y
\]
yields also a continuous family $\{(Y,\wt{\w}_s,Y^+_s,Y^-_s)\}$ of symplectic splittings. Moreover,
we assume that $\dim(\la_s\cap \mu_s)-\dim(\la_s\cap \mu_s\cap Y)$ is constant and $(\la_s\cap
Y,\mu_s\cap Y)\in \Ff\Ll^2(Y,\wt{\w}_s)$ of index $0$ for all $s$, and that the pairs define also a
continuous curve in $Y$. Then we have
\[
\Mas\{\la_s,\mu_s;P_s\} \ =\  \Mas\{\la_s\cap Y,\mu_s\cap
Y;\wt{P}_s\},
\]
where $P_s$ and $\wt P_s$ denote the projections of $X$ onto $X_s^+$
along $X_s^-$ and the projections of $Y$ onto $Y_s^+$ along $Y_s^-$
respectively.
\end{lemma}

The lemma is an immediate consequence of Lemma \ref{l:sf-embedding}
of the Appendix.

\begin{acknowledgements} We would like to thank Prof. K. Furutani (Tokyo),
Prof. M. Lesch (Bonn) and Prof. R. Nest (Copenhagen) for inspiring discussions about this subject
and BS H. Larsen (Roskilde) for drawing the figures of the Appendix. Not least we thank the referee
for careful reading and an heroic effort comprising 191 thoughtful comments, corrections, and
helpful suggestions which led to many improvements of the presentation. The referee clearly went
beyond the call of duty, and we are indebted.
\end{acknowledgements}



\appendix

\setcounter{secnumdepth}{2}

\setcounter{theorem}{0}

\setcounter{equation}{0}



\addtocontents{toc}{\medskip\noi}
\section{Spectral flow}\label{s:appendix1}

The spectral flow for a one parameter family of linear self-adjoint Fredholm operators was
introduced by M. Atiyah, V. Patodi, and I. Singer \cite{AtPaSi75} in their study of index theory on
manifolds with boundary. Since then other significant applications have been found. Later this
notion was made rigorous for curves of bounded self-adjoint Fredholm operators in J. Phillips
\cite{Ph96} and for gap-continuous curves of self-adjoint (generally unbounded) Fredholm operators
in Hilbert spaces in \cite{BoLePh01} by the Cayley transform. The notion was generalized to the
higher dimensional case in X. Dai and W. Zhang \cite{DaZh98} for Riesz-continuous families, and to
more general operators in \cite{Wo85,Zh01,ZhLo99}.

For manifolds with singular metrics, there may appear linear relations (cf. C. Bennewitz
\cite{Be72} and M. Lesch and M. Malamud \cite{LeMal03}). It is well known that many statements on
{\em relations} can be translated into those on the resolvents in the realm of {\em operator}
theory, see, e.g., B. M.  Brown, G. Grubb, and I. G. Wood \cite{BrGrWo09}. It seems to us, however,
that this translation can not always be made globally, i.e., not for a whole curve of relations.

In this Appendix we shall provide a rigorous definition of the {\it spectral flow} of {\it
spectral-continuous} curves of {\it admissible} closed linear relations in Banach spaces relative
to a co-oriented real curve  $\ell\<\C$. (All the preceding terms will be explained).

\subsection{Gap between subspaces}\label{ss:gap}

Let $\Ss(X)$ denote the set of all closed subspaces of a Banach space
$X$.

\subsubsection{The gap topology}
The {\em gap} between subspaces $M,N\in\Ss(X)$ is defined by
\begin{equation}\label{e:gap}
{\hat\delta}(M,N)\ :=\ \mmax\{\delta(M,N),\delta(N,M)\},
\end{equation}
where $\delta(M,N):=\sup\{\dist(x,N)\mid x\in M,\|x\|=1\}$, $\delta(M,\{0\}):=1$ for $M\ne\{0\}$,
and $\delta(\{0\},N):=0$. The sets $U(M,\varepsilon)=\{N\in\Ss(X)\mid \delta(M,N)<\varepsilon\}$,
where $M\in\Ss(X)$ and $\varepsilon>0$, form a basis for the so-called {\em gap topology} on
$\Ss(X)$. This is a complete metrizable topology on $\Ss(X)$ \cite[Section IV.2.1]{Ka76}.

Let $X$ be a Hilbert space. 
Then the gap between closed subspace
$M,N$ is a metric for $\Ss(X)$ and can be calculated by
\begin{equation}\label{e:gap-hilbert}
\hat\delta(M,N)\ =\ \|P_M-P_N\|,
\end{equation}
where $P_M,P_N$ denote the orthogonal projections of $X$ onto $M,N$
respectively, \cite[Theorem I.6.34]{Ka76}.

We have the following lemma.

\begin{lemma}\label{l:sub-quotion}Let $X$ be a Hilbert space, and $Y$
be a closed linear subspace of $X$. Then the mapping $M\mapsto M+Y$ induces a bijection from the
space $\Ss(X,Y)$ of closed linear subspaces of $X$ containing $Y$ onto the space
$\Ss(X/Y)=\Ss(Y^\perp)$ of closed linear subspaces of $X/Y$, which preserves the metric.
\end{lemma}

\begin{proof} We view $X/Y$ as $Y^{\bot}$. Let $M,N\subset Y^{\bot}$ be
two closed subspaces and $P_M,P_N$ be the orthogonal projections
onto $M$, $N$ respectively. Then we have
$$\hat\delta(M+Y,N+Y)\ =\ \|P_{M+Y}-P_{N+Y}\|\ =\ \|P_M-P_N\|\ =\ \hat\delta(M,N).$$
\end{proof}

\subsubsection{Uniform properties}
In general, the distances $\d(M,N)$ and $\d(N,M)$ can be very different and, even worse, behave
very differently under small perturbations. However, for finite-dimensional subspaces of the same
dimension in a Hilbert space we can estimate $\d(M,N)$ by $\d(N,M)$ in a uniform way.

\begin{lemma}\label{l:estimate-delta}Let $X$ be a Hilbert space and $M,N$ be two
subspaces with $\dim M=\dim N=n\in\NN$. If
$\delta(N,M)<\frac{1}{\sqrt{n}}$, then we have
\begin{equation}\label{e:estimate-delta}\delta(M,N)
\le\frac{\sqrt{n}\,\delta(N,M)}{1-\sqrt{n}\,\delta(N,M)}.
\end{equation}
\end{lemma}

\begin{proof}Let $y_1,\ldots,y_n$ be an orthonormal basis of $N$.
Let $x_k\in M$ denote the vectors with $\|x_k-y_k\|=\dist(y_k,M)$. Then
$\|x_k-y_k\|\le\delta(N,M)$.

For any $a_1,\ldots,a_n\in\C$, set $x=\sum_{k=1}^na_kx_k$. Then we
have
\begin{eqnarray}\|x\|&\ =\ &\|\sum_{k=1}^na_ky_k+\sum_{k=1}^na_k(x_k-y_k)\|
\ge\|\sum_{k=1}^na_ky_k\|-\sum_{k=1}^n|a_k|\|x_k-y_k\|\nonumber\\
&\ge&(\sum_{k=1}^na_k^2)^{\frac{1}{2}}-\sum_{k=1}^n|a_k|\,\delta(N,M)
\ge(1-\sqrt{n}\,\delta(N,M))(\sum_{k=1}^na_k^2)^{\frac{1}{2}}.\label{e:x-delta}
\end{eqnarray}
If $x=0$, by (\ref{e:x-delta}) we have $a_k=0$. Thus
$x_1,\ldots,x_n$ are linearly independent and therefore they form a
basis of $M$.

For any $x=\sum_{k=1}^n a_kx_k\in M$ with $\|x\|=1$, let $y=\sum_{k=1}^na_ky_k$. By
(\ref{e:x-delta}) we have
\begin{equation*}\|x-y\|\ =\ \|\sum_{k=1}^na_k(x_k-y_k)\|
\le\sum_{k=1}^n|a_k|\,\delta(N,M)
\le\frac{\sqrt{n}\,\delta(N,M)}{1-\sqrt{n}\,\delta(N,M)}.
\end{equation*}
Hence we have (\ref{e:estimate-delta}).
\end{proof}

Clearly, taking the sum of two closed subspaces is not a continuous operation in general, but
becomes continuous when fixing the dimension of the intersection and keeping the sum closed.

The following Lemma is well-known and the proof is omitted.

\begin{lemma}\label{l:continuity-ker-range}Let $X,Y$ be two Hilbert space
and $A_s\in\Bb(X,Y)$ be a norm-continuous family of semi-Fredholm bounded operators. If $\dim\ker
A_s$ is constant, then $\ker A_s\in\Ss(X)$ and $\range A_s\in\Ss(Y)$ are continuous families of
closed subspaces (continuous in the gap topology).
\end{lemma}

We recall the notion of semi--Fredholm pairs:
Let $M,N\in\Ss(X)$. The pair $M,N$ is called {\em (semi-)Fredholm}
if $M+N$ is closed in $X$, and both of (one of) the spaces $M\cap N$
and $\dim X/(M+N)$ are (is) finite dimensional. In this case, the
{\em index} of $(M,N)$ is defined by
\begin{equation}\label{e:index-pair}\index(M,N)\ :=\ \dim M\cap N-\dim
X/(M+N)\ \in\ \ZZ\cup\{-\infty,\infty\}.
\end{equation}
Note that by \cite[Remark A.1]{BoFu99} (see also \cite[Problem 4.4.7]{Ka76}),
$X/(M+N)$ of finite
dimension implies $M+N\in\Ss(X)$.

\begin{proposition}\label{p:continuous-intersection-sum}
Let $X$ be a Hilbert space and $n\in \NN$. Denote by ${\mathcal {SF}}^2_{1,n}(X)$
{\rm(}respectively ${\mathcal {SF}}^2_{2,n}(X)${\rm)} the set of semi-Fredholm pairs $(M,N)$ of
closed subspaces with $\dim M\cap N=n$ {\rm(}respectively $\dim X/(M+N)=n${\rm)}. Then the
following four natural mappings $\varphi_{k,l}\colon  \Ss\Ff^2_{l,n}(X)\to\Ss(X)$, $k,l=1,2$ are
continuous:
\[\varphi_{1,l}(M,N)\ :=\ M\cap N,\qquad
\varphi_{2,l}(M,N)\ :=\ M+N.
\]
\end{proposition}

\begin{proof} (Communicated by R. Nest) Let $(M,N)\in\Ss(X)\times\Ss(X)$.
Let $P_M$ and $P_N$ denote the orthogonal projections of $X$ onto
$M$ and $N$ respectively. Then we have
\[
\range P_M+\range P_N\ =\ \range ((I-P_N)P_M)+\range P_N\,.
\]
So $M+N$ is closed if and only if $\range ((I-P_N)P_M)$ is closed, and the kernel of
$(I-P_N)P_M\in\Bb(\ran P_M,\ker P_N)$ is $M\cap N$. By Lemma \ref{l:continuity-ker-range}, the maps
$\varphi_{k,1}$, $k=1,2$ are continuous. Recall that taking orthogonal complements is continuous.
Then $\varphi_{k,2}$ is continuous by the fact that
\[\varphi_{k,2}(M,N)\ =\ \big(\varphi_{3-k,1}(M^\perp,N^\perp)\big)^\perp,\quad k=1,2.\]

\end{proof}

\subsection{Closed linear relations}\label{ss:clr}
This subsection discusses some general properties of closed linear
relations. For additional details, see Cross \cite{Cr98}.

\subsubsection{Basic concepts of closed linear relations}
Let $X,Y$ be two vector spaces. A {\em linear relation} $A$ between $X$ and $Y$ is a linear
subspace of $X\times Y$. As usual, the {\em domain}, the {\em range}, the {\em kernel} and the {\em
indeterminant part} of $A$ are defined by
\begin{eqnarray*}\dom(A)&\ :=\ &\{x\in X\mid \;\mbox{there exists}\;y\in Y
\;\mbox{such that}\; (x,y)\in A\}, \\
\ran A&\ :=\ &\{y\in Y\mid \;\mbox{there exists}\;x\in X \;\mbox{such
that}\; (x,y)\in A\},\\
\ker A&\ :=\ &\{x\in X\mid\;(x,0)\in A\},\\
A(0)&\ :=\ &\{y\in Y\mid\;(0,y)\in A\},
\end{eqnarray*}
respectively.

Let $X,Y,Z$ be three vector spaces. Let $A,B$ be linear relations
between $X$ and $Y$, and $C$ a linear relation between $Y$ and $Z$.
We define $A+B$ and $CA$ by
\begin{eqnarray}
A+B\ &:=&\ \{(x,y+z)\in X\times Y\mid\;(x,y)\in A, (x,z)\in B\},\label{e:a-add}\\
CA\ &:=&\ \{(x,z)\in X\times Z\mid\;\exists y\in Y\;\mbox{such that}\; (x,y)\in A, (y,z)\in
C\}\label{e:a-comp}.
\end{eqnarray}

\begin{definition}\label{d:clr} Let $X,Y$ be two Banach spaces. A
{\em closed linear relation} between $X,Y$ is a closed linear
subspace of $X\times Y$. We denote by $\CLR(X,Y)=\Ss(X\times Y)$ and
$\CLR(X)=\Ss(X\times X)$.
\end{definition}

Note that a linear relation $A$ between $X,Y$ is a graph of a linear operator if and only if
$A(0)=\{0\}$. In this case we shall still denote the corresponding operator by $A$. After
identifying an operator and its graph, we have the inclusions
$$\Bb(X,Y)\subset\Cc(X,Y)\subset\CLR(X,Y)$$
with the notations of Section \ref{ss:isometric}.

Let $A$ be a linear relation between $X,Y$. The {\em inverse}
$A^{-1}$ of $A$ is a always defined. It is the linear relation between $Y,X$ defined by
\begin{equation}\label{e:inverse-clr}
A^{-1}=\{(y,x)\in Y\times X;(x,y)\in A\}.
\end{equation}

\begin{definition}\label{d:fredholm-invertible} Let $X,Y$ be two Banach
spaces and $A\in\CLR(X,Y)$.
\begin{enumerate}
\item $A$ is called {\em Fredholm}, if $\dim\ker A<+\infty$,
$\ran A$ is closed in $Y$ and $\dim (Y/\ran A)<+\infty$. In this case, we define the {\em index} of
$A$ to be
\begin{equation}\label{e:index-clr}\index A\ =\ \dim\ker A-\dim (Y/\ran A).\end{equation}
\item $A$ is called {\em bounded invertible}, if $A^{-1}\in\Bb(Y,X)$.
\end{enumerate}\end{definition}

\begin{lemma}(a) $A$ is Fredholm, if and only if the pair $(A,X\times\{0\})$ is
a Fredholm pair of closed subspaces of $X\times Y$. In this case, $\index
A=\index(A,X\times\{0\})$.

\noi (b) $A$ is bounded invertible, if and only if $X\times X$ is the
direct sum of $A$ and $X\times\{0\}$.
\end{lemma}

\begin{proof} Our results follow from the fact that
\[
A\cap (X\times\{0\})\ =\ \ker A\times\{0\},\ \tand\ A+\left(X\times\{0\}\right)\ =\
\left(\{0\}\times\range A\right)+\left(X\times\{0\}\right).
\]
\end{proof}

\subsubsection{Spectral projections of closed linear relations}
\begin{definition}\label{d:spectral-clr}Let $X$ be a Banach space
and $A\in\CLR(X)$. Let $\zeta$ be a complex number. $\zeta$ is
called a {\em regular point} of $A$ if $A-\zeta I$ is bounded
invertible. Otherwise $\zeta$ is called a {\em spectral point} of $A$. We
denote the set of all spectral points of $A$ by $\sigma(A)$ and
the set of all regular points of $A$ by $\rho(A)$. The {\em resolvent} of
$A$ is defined by
\begin{equation}\label{e:resolvent}R(\zeta,A)\ =\ (A-\zeta I)^{-1},\quad
\zeta\in\rho(A).\end{equation}
\end{definition}

Let $X$ be a Banach space, and $A\in\CLR(X)$. Let $N\<\C$ be a bounded open subset. Assume that
$\sigma(A)\cap\partial N$ does not contain an accumulation point of $\sigma(A)\cap N$. Then there
exists an open subset $N_1\subset N$ such that
\begin{equation}\label{e:a-rev-admissible}
\ol{N_1}\< N,\
\partial N_1 \in C^1\/,\  \s(A)\cap N_1\ =\ \s(A)\cap N,
\quad\tand\quad \s(A)\cap\partial N_1\ =\ \emptyset,
\end{equation}
and the spectral projection
\begin{equation} \label{P_N}
P_N(A) \ :=\  -\frac 1{2\pi i}\int_{\partial N_1}(A-\zeta I)^{-1} d\zeta
\end{equation}
is well-defined and does not depend on the choice of $N_1$. We have the following lemma (cf.
\cite[Theorem III.6.17]{Ka76}):

\begin{lemma} \label{l:spectral-check}(a) We have
\begin{equation}\label{e:PA1}P_N(A)A\subset AP_N(A)\ =\ P_N(A)AP_N(A)+\{0\}\times
A(0),\end{equation}
where the composition is taken in the sense of \eqref{e:a-comp}.

\noi (b) We have
\begin{equation}\label{e:PA2}P_N(A)AP_N(A)\ =\ -\frac 1{2\pi i}\int_{\partial N_1}\zeta(A-\zeta
I)^{-1} d\zeta.
\end{equation}

\noi (c) If we view $P_N(A)AP_N(A)$ as a linear relation on $\ran(P_N(A))$, then we have
$P_N(A)AP_N(A)\in\Bb\bigl(\ran(P_N(A))\bigr)$, and
\begin{equation}\label{e:spectral-check} \s(A)\cap N \ =\  \s\bigl(
P_N(A)AP_N(A)\bigr).
\end{equation}
\end{lemma}

\begin{proof}
Let $z\in N\setminus N_1$ be a regular point. Then we have
\[P_N(A)R(z,A)=R(z,A)P_N(A)=-\frac 1{2\pi i}\int_{\partial N_1}(z-\zeta)^{-1}(A-\zeta
I)^{-1} d\zeta.\] Since $R(z,A)$ is bounded and $0\ne (z-\zeta)^{-1}$ for all
$\zeta\in\sigma(A)\cap N_1$, we have $\ker R(z,A)=A(0)\subset\ker(P_N(A))$. Then our results follow
form the corresponding results for $R(z,A)$.
\end{proof}

\subsection{Spectral flow for closed linear relations.}\label{ss:A3}
At first we give the definition of admissible relations.

\begin{figure}[tb!]
\includegraphics{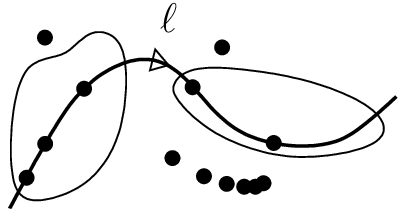}
\includegraphics{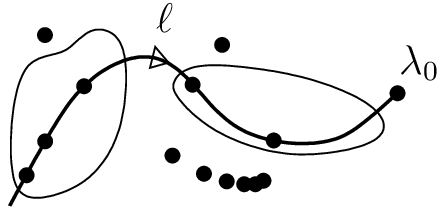}

\bigskip

\includegraphics{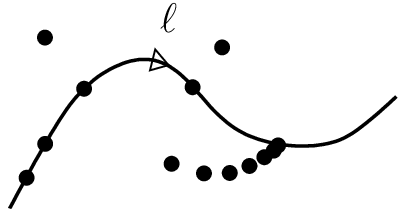}

\caption{Upper left: Closed linear relation with admissible spectrum with respect to $\ell$. Upper
right: Admissible spectrum with $\la_0\in \ol{\ell}\setminus\ell$\,. Bottom: Non-admissible
spectrum since $\s(A)\cap N\ne\s(A)\cap\ell$ and $\dim\ran P_N(A)=+\infty$, each contradicting
\eqref{e:a-admissible}(i) and (ii)\label{f:1}}
\end{figure}

\begin{figure}[tb!]
\includegraphics{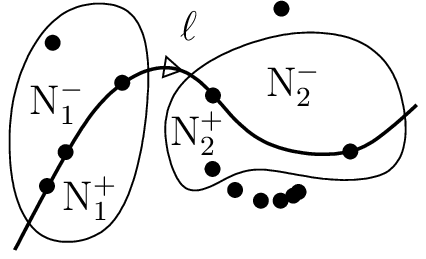}

\bigskip

\includegraphics{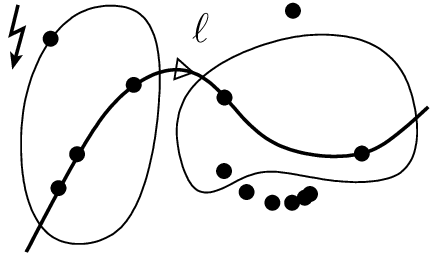}
\includegraphics{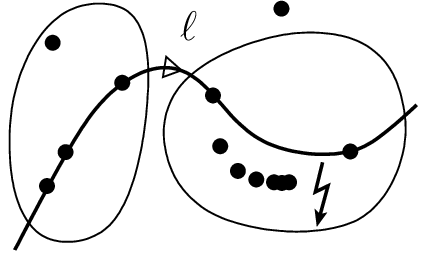}

\bigskip

\includegraphics{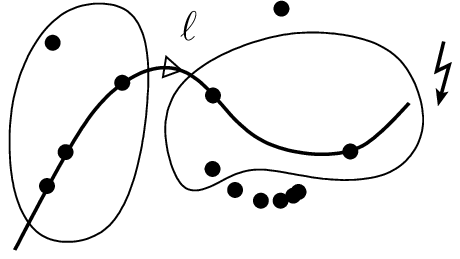}
\includegraphics{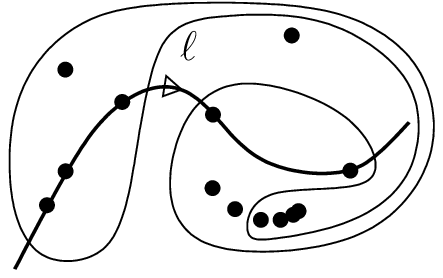}

\caption{Top: Admissible test domain triple $(N, N^+,N^-)$. Middle and bottom: Non-admissible test
domain triples. Middle left: $\s(A)\cap\partial N\ne\emptyset$\,. Middle right: $\dim
P_N(A)=+\infty$\,. Bottom left: $\ol{N^0}=\ol{N}\cap\ell$ not satisfied. Bottom right: $N\cap\ell$
not connected while $N$ connected\label{f:2}}
\end{figure}

\begin{definition}\label{d:admissible} (Cf. Zhu \cite[Definition 1.3.6]{Zh00},
\cite[Definition 2.1]{Zh01}, and \cite[Definition 2.6]{ZhLo99}). Let
$\ell\<\C$ be a $C^1$ real 1-dimensional submanifold which has no
boundary and is co-oriented (i.e., with oriented normal bundle). Let
$X$ be a Banach space and $A\in\CLR(X)$ be a closed linear relation.

\noi (a) We call $A$ {\em admissible} with respect to $\ell$, if
there exists a bounded open subset $N$ of $\C$ (called {\em test domain})
such that (see also Fig. \ref{f:1})
\begin{equation}\label{e:a-admissible}
\text{(i) }\s(A)\cap N\ =\ \s(A)\cap\ell\quad \tand\quad \text{(ii) }\dim\ran P_N(A)<+\infty\/.
\end{equation}
Then $P_N(A)$ does not depend on the choice of such a test domain $N$.
We set
\begin{equation}\label{e:nullity}
P_{\ell}(A)\ :=\  P_N(A) \quad\tand\quad \nu_{\ell}(A)\ :=\ \dim\ran P_N(A).
\end{equation}
For fixed $\ell$ and $X$ we shall denote the space of all
$\ell$-admissible closed linear relations in $X$ by $\Aa_{\ell}(X)$.

\smallskip

\noi (b) Let $A\subset\Aa_{\ell}(X)$. Let $N\subset\C$ be open and bounded with $C^1$ boundary. We
set $N^0:=N\cap\ell$ and assume
\begin{equation}
\ol{N^0}\subset\ell,\ \quad \sigma(A)\cap\ell\subset N,\ \quad \sigma(A)\cap\partial N=\emptyset,\
\quad{\rm and}\quad \dim\image P_N(A)<+\infty.
\end{equation}
Moreover, we require that $\partial N$ intersects $\ell$ transversely and that each connected
component of $N$ has connected intersection with $\ell$. Then
\begin{itemize}
\item $\ol{N^0}=\ol{N}\cap\ell$,

\item the positive and negative parts $N^{\pm}$ of $N$ with respect to the co-orientation of
$\ell$ are well-defined, and

\item we have a disjoint union $N=N^+\cup N^0\cup N^-$\/.

\end{itemize}
We shall call the resulting triple $(N;N^+,N^-)$ {\em admissible} with respect to $\ell$ and $A$,
and write $(N;N^+,N^-)\in\Aa_{\ell,A}$. Clearly the set $\Aa_{\ell,A}$ is non-empty. See also Fig.
\ref{f:2}.
\end{definition}

\begin{note}
To prove $\ol{N^0}=\ol{N}\cap\ell$, we notice $\ol{N^0}\subset\ell$. So we have
$\ol{N^0}\subset\ol{N}\cap\ell$. Since $\partial N$ intersects $\ell$ transversely, we have
$\partial N\cap\ell\subset\ol{N^0}$. Then $\ol{N}\cap\ell\subset \ol{N^0}$. That yields
$\ol{N^0}=\ol{N}\cap\ell$.
\end{note}

Now we are able to define spectral-continuity and the spectral flow. Our data are a co-oriented
curve $\ell\<\C$, a family of Banach spaces $\{X_s\}_{s\in [a,b]}$\, and a family $\{A_s\}_{s\in
[a,b]}$ of closed linear relations on $X_s$\/.

\begin{figure}[tb!]
\includegraphics{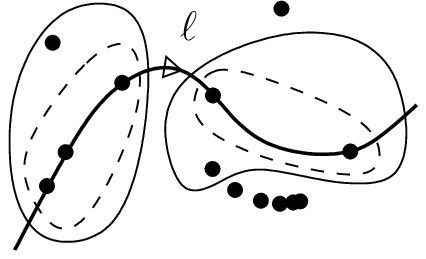}

\bigskip

\includegraphics{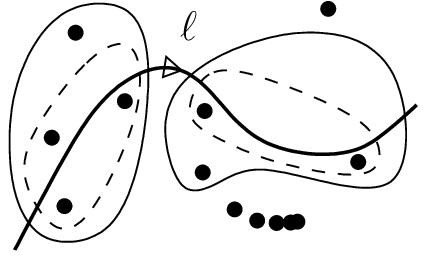}
\includegraphics{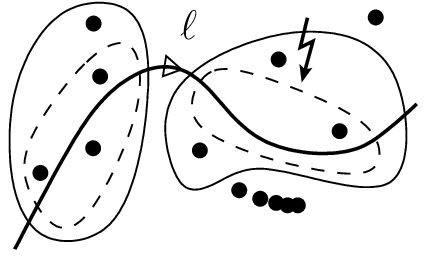}

\caption{Neighborhoods of the spectra of a spectral-continuous family near $\ell$ at $s_0$\,: The
same test domain triple $(N, N^+,N^-)$ (solid line) works at $s_0$ in the upper figure, at $s_0-\e$
in bottom left, and at $s_0+\e$ in bottom right. The sub-triple $(N', N'^+,N'^-)$ (encircled by the
broken line) will also work at $s_0$ and for $s_0-\e$, but only for $s_0+\e'$ with
$\e'\ll\e$\label{f:3}}
\end{figure}

\begin{definition}\label{d:spectral-continuous}
(a) We shall call the family $\left\{A_s\in\Aa_{\ell}(X_s)\right\}$, $s\in[a,b]$ {\em spectral
continuous} near $\ell$ at $s_0\in[a,b]$, if
\begin{enumerate}

\item there is an $\e(s_0)>0$ and a triple $(N;N^+,N^-)$ such that
$$(N;N^+,N^-)\in\Aa_{\ell,A_s}\quad\text{for all } |s-s_0|<\e(s_0),$$

\item for all triple $(N^{\prime};N^{\prime +},N^{\prime -})\in\Aa_{\ell,A_{s_0}}$ with
$N^{\prime}\subset N$ and $N^{\prime\pm}\subset N^{\pm}$, we have
$$(N^{\prime};N^{\prime +},N^{\prime -})\in\Aa_{\ell,A_s}\quad \text{for all } |s-s_0|\ll 1,$$

\item for all triple $(N^{\prime};N^{\prime +},N^{\prime -})$ and subinterval $K$ of
$\left(s_0-\e(s_0),s_0+\e(s_0)\right)$ with $N^{\prime}\subset N$, $N^{\prime\pm}\subset N^{\pm}$,
and $(N^{\prime};N^{\prime +},N^{\prime -})\in\Aa_{\ell,A_s}$ for all $s\in K$, we have $\dim\image
P_{N^{\prime}}(A_s)$ and $\dim\image P_{N^{\pm}\setminus N^{\prime\pm}}(A_s)$ do not depend on
$s\in K$. See also Fig. \ref{f:3} and Fig. \ref{f:4}.
\end{enumerate}

\begin{figure}[tb!]
\includegraphics{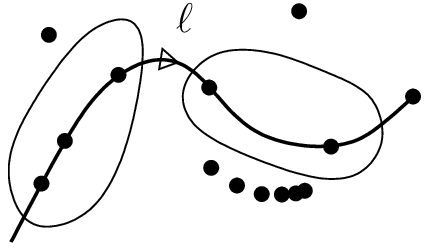}
\includegraphics{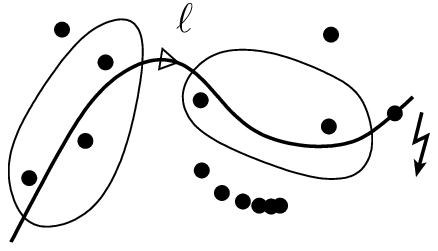}

\caption{A curve of closed linear relations with admissible spectra may fail to become
spectral-continuous near $\ell$ due to a spectral point $\la_0\in\ol{\ell}\setminus \ell$ for
$s_0$\, (left), which moves inward on $\ell$ for $s=s_0\pm\e$ (right)\label{f:4}}
\end{figure}

\noi We shall call the family $\{A_s\}\in \Aa_{\ell}(X_s)$, $s\in
[a,b]$ {\em spectral-continuous} near $\ell$, if it is
spectral-continuous near $\ell$ at $s_0$ for all $s_0\in [a,b]$.

\smallskip

\noi (b) Let $\{A_s\}\in\Aa_{\ell}$, $s\in [a,b]$  be a family of admissible operators that is
spectral-continuous near $\ell$. Then there exists a partition
\begin{equation}\label{e:partition}
a\ =\ s_0\leq t_1\leq s_1\leq\ldots s_{n-1}\leq t_n\leq s_n\ =\ b
\end{equation}
of the interval $[a,b]$, such that $s_{k-1},s_k\in (t_k-\e(t_k),t_k+\e(t_k))$, $k=1,\ldots,n$. Let
$(N_k;N_k^+,N_k^-)$ be like a $(N;N^+,N^-)$ in (a) for $t_k$ such that
$s_{k-1},s_k\in(t_k-\e(t_k),t_k+\e(t_k)),k=1,\ldots,n$.  Then we define the {\em spectral flow} of
$\{A_s\}_{a\leq s\leq b}$ through $\ell$ by
\begin{equation}\label{e:spectral-flow}
\SF_{\ell}\bigl\{A_s; a\leq s \leq b\bigr\} \ :=\
\sum_{k=1}^n\Bigl(\dim\ran\bigl(P_{N_k^-}(A_{s_{k-1}})\bigr) -
\dim\ran\bigl(P_{N_k^-}(A_{s_k})\bigr)\Bigr) .
\end{equation}
When $\ell$ is a bounded open submanifold of $i\R$ containing 0 with co-orientation from left to
right, we set
$$\SF\bigl\{A_s; a\leq s \leq b\bigr\}
\ :=\ \SF_{\ell}\bigl\{A_s; a\leq s \leq b\bigr\}.$$
\end{definition}

From our assumptions it follows that the spectral flow is independent of the choice of the
partition (\ref{e:partition}) and admissible $(N_k;N_k^+,N_k^-)$, hence it is well-defined. From
the definition it follows that the spectral flow through $\ell$ is path additive under catenation
and homotopy invariant. For details of the proof, see \cite{Ph96} and \cite{ZhLo99}.

\begin{lemma}\label{l:16} Let ${\ell\subset\C}$ be as in Definition \ref{d:admissible}a and $X$ be a Banach space.
Let $\{A_s\in\CLR_{\ell}(X)\}$, $a\le s\le b$ be a continuous family and
$A:=A_{s_0}\in\Aa_{\ell}(X)$ with $s_0\in[a,b]$. Let ${\mathfrak{m}}$ be a bounded open submanifold
of $\ell$ such that $\ol{\mathfrak{m}}\subset\ell$. If $\sigma(A_s)\cap\ell\subset{\mathfrak{m}}$
for all $s\in[a,b]$, we have:

(a) There exists an $\e>0$ such that $A_s\in\Aa_{\ell}(X)$ for all $s\in(s_0-\e,s_0+\e)$.

(b) The family $\{A_s\}$ is spectral continuous near $\ell$ at $s_0$.
\end{lemma}

\begin{proof} \textit{a}.  Since $\ell$ is co-oriented and $A\in\Aa_{\ell}(X)$, there exists a bounded
open subset $N$ of $\C$ such that $\sigma(A)\cap N=\sigma(A)\cap\ell$ and $\dim\image
P_N(A)<+\infty$. Since ${\mathfrak{m}}$ is a bounded open submanifold of $\ell$,
$\ol{\mathfrak{m}}\subset\ell$ and $\sigma(A)\cap\ell\subset{\mathfrak{m}}$, we can choose $N$ such
that $\partial N$ is $C^1$, $N\cap\ell={\mathfrak{m}}$ and $\sigma(A)\cap\partial N=\emptyset$.
Since $\{A_s\in\CLR_{\ell}(X)\}$, $a\le s\le b$ is a continuous family and $\partial N$ is compact,
there exists an $\e>0$ such that for all $s\in(s_0-\e,s_0+\e)$, $\sigma(A_s)\cap N=\emptyset$. Then
$\{P_N(A_s)\}$, $|s-s_0|<\e$ is a well-defined continuous family of projections on $X$ and
$\dim\image P_N(A_s)\le\dim\image P_N(A)<+\infty$. We also have
$\sigma(A_s)\cap\ell\subset{\mathfrak{m}}\subset N$. So there exists an open subset $N_s$ of $N$
such that $\sigma(A_s)\cap\ell=\sigma(A_s)\cap N_s$, and $\dim\image P_{N_s}(A_s)\le \dim\image
P_N(A_s)<+\infty$. (a) is proved.

\noi\textit{b}. Since $A\in\Aa_{\ell}(X)$, there exists a triple $(N;N^+,N^-)\in\Aa_{\ell,A}$. As
in the proof of (a), we can choose $N$ such that $N\cap\ell={\mathfrak{m}}$ and $\sigma(A)\cap
N=\sigma(A)\cap\ell$. Let $\e$ be as in the proof of (a). If $|s-s_0|<\e$, we have
$\sigma(A_s)\cap\ell\subset{\mathfrak{m}}\subset N$, $\sigma(A_s)\cap\partial N=\emptyset$ and
$\dim\image P_N(A_s)<+\infty$. So $(N;N^+,N^-)\in\Aa_{\ell,A_s}$ for $|s-s_0|<\e$.

Given a triple $(N^{\prime};N^{\prime +},N^{\prime -})\in\Aa_{\ell,A}$ with $N^{\prime}\subset N$
and $N^{\prime\pm}\subset N^{\pm}$, we have
$\overline{N^{\prime}\cap\ell}=\overline{N\cap\ell}\subset\ol{\mathfrak{m}}\subset\ell$ and
$\sigma(A)\cap\partial N^{\prime}=\emptyset$. Since $\partial N^{\prime}$ is compact, for
$|s-s_0|\ll 1$, we have $\sigma(A_s)\cap\partial N^{\prime}=\emptyset$. Then $\dim\image
P_{N^{\prime}}(A_s)$ does not depend on $s$ and is finite. Since $\sigma(A)\cap
N=\sigma(A)\cap\ell\subset\sigma(A)\cap N^{\prime}$, we
have $\sigma(A)\cap N=\sigma(A)\cap N^{\prime}$. By%
\[
\dim\image P_{N^{\prime}}(A_s)\ =\ \dim\image P_{N^{\prime}}(A)\ =\ \dim\image P_N(A)\ =\
\dim\image P_N(A_s),%
\]
we get $P_{N^{\prime}}(A_s)=P_N(A_s)$. So we have $\sigma(A_s)\cap\ell\subset\sigma(A_s)\cap
N=\sigma(A_s)\cap N^{\prime}\subset N^{\prime}$ and $(N^{\prime};N^{\prime +},N^{\prime
-})\in\Aa_{\ell,A_s}$.

Given an interval $K\subset\left(s_0-\e(s_0),s_0+\e(s_0)\right)$ and a triple
$(N^{\prime};N^{\prime +},N^{\prime -})\in\Aa_{\ell,A_s}$ with $N^{\prime}\subset N$,
$N^{\prime\pm}\subset N^{\pm}$ and $s\in K$, we have $\sigma(A_s)\cap\partial
N=\sigma(A_s)\cap\partial N^{\prime}=\emptyset$ for all $s\in K$. Since $\sigma(A_s)\cap\ell\subset
N$ and $\sigma(A_s)\cap\ell\subset N^{\prime}$, we have $\sigma(A_s)\cap\left(\ell\setminus
N^{\prime}\right)=\emptyset$ and $\sigma(A_s)\cap\left({\mathfrak{m}}\setminus
N^{\prime}\right)=\emptyset$. Define the closed curve%
\[
C^{\pm}\ :=\ (\partial N\cap N^{\pm})\cup({\mathfrak{m}}\setminus N^{\prime})\cup(\partial N\cap
N^{\prime\pm})%
\]
with the orientation of $\partial N\cap N^{\pm}$ and opposite orientation of $\partial N\cap
N^{\prime\pm}$.
Then we have%
\[
P_{N^{\pm}\setminus N^{\prime\pm}}(A_s)\ =\ -\frac{1}{2\pi i}\int_{C^{\pm}}(A_s-\zeta I)^{-1}\zeta.%
\]
The families of projections $\left\{P_{N^{\prime}}(A_s)\right\}$ and $\{P_{N^{\pm}\setminus
N^{\prime\pm}}(A_s)\}$, $s\in K$ are continuous. So we have $\dim\image P_{N^{\prime}}(A_s)$ and
$\dim\image P_{N^{\pm}\setminus N^{\prime\pm}}(A_s)$ do not depend on $s\in K$. (b) is proved.
\end{proof}

\medskip

We close the appendix by discussing the invariance of the spectral flow under embeddings in larger
spaces, assuming a simple regularity condition.

\begin{lemma}\label{l:sf-embedding}
Let $\{Y_s; s\in [a,b]\}$ and $\{X_s; s\in [a,b]\}$ be two families of (complex) Banach spaces with
$X_s\< Y_s$ (no density or continuity of the embeddings are assumed). Let $\{A_s\in\CLR(Y_s); s\in
[a,b]\}$ be a spectral-continuous curve near a fixed co-oriented curve $\ell\<\C$. We assume that
$A_s(X_s)\< X_s$ for all $s$ and that the curve%
\[
\{A_s|_{X_s}\in\CLR(X_s); s\in [a,b]\}%
\]
is also spectral-continuous near $\ell$. Then we have
\[
\SF_{\ell}\{A_s;s\in [a,b]\}\ =\  \SF_{\ell}\{A_s|_{X_s};s\in
[a,b]\}
\]
if the difference $\dim\nu_{\ell}(A_s)-\dim\nu_{\ell}(A_s|_{X_s})$, $s\in [a,b]$, is a constant
$m$. In this case, $m\ge 0$.
\end{lemma}


\begin{proof} We go back to the local definition of $\SF_{\ell}$ and reduce to the
finite-dimensional case. Let $s_0\in [a,b]$. Choose a triple
\[
(N_1;N_1^+,N_1^-)\in \Aa_{\ell,A_{s_0}}
\]
such that $N_1$ satisfies \eqref{e:a-admissible} for $A_{s_0}$ and $A_{s_0}|_{X_{s_0}}$. By
\eqref{e:a-admissible} we have
$$P_{N_1}(A_{s_0})\ =\ \nu_{\ell}(A_{s_0})\quad {\rm and}\quad
P_{N_1}(A_{s_0}|_{X_{s_0}})\ =\ \nu_{\ell}(A_{s_0}|_{X_{s_0}}).$$  Then by spectral continuity,
there exists a triple $(N;N^+,N^-)$ with $\ol{N}\subset N_1$ with
$$(N;N^+,N^-)\in\Aa_{\ell,A_s}\cap\Aa_{\ell,A_s|_{X_s}}\quad {\rm for}\;|s-s_0|\ll 1.$$
Then we have
$$P_{N_1}(A_{s_0})\ =\ \nu_{\ell}(A_{s_0})\ =\ P_N(A_{s_0})\quad {\rm and}\quad P_{N_1}(A_{s_0}|_{X_{s_0}})
\ =\ \nu_{\ell}(A_{s_0}|_{X_{s_0}})=P_N(A_{s_0}|_{X_{s_0}}),$$ and for $\abs{s-s_0}\ll 1$
\begin{equation}\label{e:nus}
\dim\range P_N(A_s)\ =\ \nu_{\ell}(A_{s_0})\ =\ \nu_{\ell}(A_{s_0}|_{X_{s_0}})+m \ =\  \dim\range
P_N(A_s|_{X_s})+m
\end{equation}
by spectral-continuity and our assumption. Now we consider for each
$\la\in\C\cap N$ the algebraic multiplicities and find
\begin{equation}\label{e:kernels}
\dim \ker (A_s|_{X_s}-\la I|_{X_s})^k \leq \dim \ker (A_s-\la I)^k
\end{equation}
for each $k\in\N$. By our assumption, we have $\nu_{\ell}(A_s)=\nu_{\ell}(A_s|_{X_s})+m$. Comparing
\begin{align*}
\dim\range P_N(A_s)&\ =\ \sum_{\la\in\s(A_s)\cap N}\ \sum_{k\in\N} \dim\ker(A_s-\la I)^k\,,\\
\dim\range P_N(A_s|_{X_s}) &\ =\  \sum_{\la\in\s(A_s|_{X_s})\cap N}\ \sum_{k\in\N}
\dim\ker(A_s|_{X_s}-\la I|_{X_s})^k\,, \\
\nu_{\ell}(A_s)&\ =\ \sum_{\lambda\in\sigma(A_s)\cap\ell} \ \sum_{k\in\N} \dim\ker(A_s-\la I)^k\,,\tand\quad\\
\nu_{\ell}(A_s|_{X_s})&\ =\ \sum_{\lambda\in\sigma(A_s|_{X_s})\cap\ell} \ \sum_{k\in\N}
\dim\ker(A_s|_{X_s}-\la I|_{X_s})^k\,,
\end{align*}
we obtain from equation \eqref{e:nus} and the inequalities \eqref{e:kernels} that $m\ge 0$ and
\[
\dim \ker (A_s|_{X_s}-\la I|_{X_s})^k\ =\ \dim \ker (A_s-\la I)^k
\]
for each $\la\in N\setminus\ell$ and $k\in\N$. So
$
\s(A_s)\cap (N\setminus\ell)\ =\ \s(A_s|_{X_s})\cap
(N\setminus\ell);
$
and the algebraic multiplicities with respect to $A_s$ and
$A_s|_{X_s}$ coincide in each point. By the definition of the
spectral flow, the two spectral flows must coincide.
\end{proof}

\addtocontents{toc}{\medskip\noi}


\begin{thebibliography}{99}


\bibitem{Am} \textsc{W. Ambrose}, The index theorem in Riemannian geometry,
{\it Ann. of Math.} {\bf 73} (1961), 49--86.



\bibitem{Ar78}
{\sc V.I. Arnold}, \textit{Mathematical Methods of Classical Mechanics},
Springer--Verlag, Graduate Texts in Mathematics vol. 60, 1978.
Title of the Russian Original Edition: {\it Matematicheskie metody klassichesko{\v i}
mekhaniki}, Nauka, Moscow, 1974.

\bibitem{At85} {\sc M.F. Atiyah}, Circular symmetry and stationary-phase
approximation, {\em in}: {\it Colloquium in Honour of Laurent
Schwartz}, Vol. 2, {\em Ast\'erisque} (1985), pp. 43--60, reprinted
in M.F. Atiyah, {\em Collected Works}, Vol. 5, Oxford University
Press, Oxford, 2005, pp. 667--685.

\bibitem{AtPaSi75} {\sc M.F. Atiyah, V.K. Patodi, and I.M.
Singer}, Spectral asymmetry and Riemannian geometry. III, {\it Math. Proc. Cambridge Phil. Soc.}
{\bf 79} (1976), 71--99.

\bibitem{Be72} {\sc C. Bennewitz}, Symmetric relations on a Hilbert space, {\em in}:
W.N. Everitt, B.D. Sleeman (eds.), {\em Conference on the Theory of Ordinary and Partial
Differential Equations} (Dundee, Scotland, 1972), Lecture Notes in Math., vol. 280,
Springer, Berlin, 1972, pp. 212--218.



\bibitem{BCLZ}\textsc{B. Boo{\ss}--Bavnbek, G. Chen, M. Lesch and C.
Zhu}, Perturbation of sectorial projections of elliptic
pseudo-differential operators, \textit{J. Pseudo-Differ. Oper.
Appl.} {\bf 3}, no. 1 (2012), 49--79, arXiv:1101.0067 [math.SP].

\bibitem{BoFu98}
{\sc B. Booss--Bavnbek and K. Furutani}, The Maslov index -- a
functional analytical definition and the spectral flow formula, {\it
Tokyo J. Math.} {\bf 21} (1998), 1--34.

\bibitem{BoFu99}
{\sc ---, ---}, Symplectic functional analysis and spectral
invariants, {\it in}: B. Boo{\ss}--Bavnbek, K.P. Wojciechowski (eds.),
``Geometric Aspects of Partial Differential Equations", Amer. Math.
Soc. Series {\it Contemporary Mathematics}, vol. 242, Providence,
R.I., 1999, pp. 53--83.

\bibitem{BoFuOt01}\textsc{B. Booss--Bavnbek, K. Furutani, and N. Otsuki},
Criss--cross reduction of the Maslov index and a proof of the
Yoshida--Nicolaescu Theorem, \textit{Tokyo J. Math.} \textbf{24}
(2001), 113--128.


\bibitem{BoLePh01}\textsc{B. Booss--Bavnbek, M. Lesch, and J. Phillips},
Unbounded Fredholm operators and spectral flow,
\textit{Canad. J. Math.} {\bf 57}/2 (2005),
225--250, arXiv: math.FA/0108014.




\bibitem{BoWo93}
{\sc B. Booss--Bavnbek and K.P. Wojciechowski}, \textit{Elliptic
Boundary Problems for Dirac Operators}, Birkh\"auser, Boston, 1993.

\bibitem{BoZh04}\textsc{B. Booss--Bavnbek and C. Zhu},
Weak symplectic functional analysis and general spectral flow formula, 2004,
arXiv:math.DG/0406139.

\bibitem{BoZh05}\textsc{---, ---},
General spectral flow formula for fixed maximal domain,
{\em Centr. Europ. J. Math.} {\bf 3}(3) (2005), 558-577.

\bibitem{BoZh10b}\textsc{---, ---},
Symplectic reduction and general spectral flow formula, {\em In preparation}.

\bibitem{BrGrWo09} \textsc{B. M.  Brown, G. Grubb, and I. G. Wood},
$M$-functions for closed extensions of adjoint pairs of operators
with applications to elliptic boundary problems, \textit{Math.
Nachr.} {\bf 282}/3 (2009), 314--347.


\bibitem{BrLe01}\textsc{J. Br\"{u}ning and M. Lesch}, On boundary value problems
for Dirac type operators. I. Regularity and self-adjointness,
\textit{J. Funct. Anal.} {\bf 185} (2001), 1--62, arXiv:math.FA/9905181.


\bibitem{CaLeMi94}{\sc S.E. Cappell, R. Lee, and E.Y. Miller},
On the Maslov index, {\it Comm. Pure Appl. Math.} {\bf 47} (1994),
121--186.

\bibitem{CaLeMi96}{\sc ---}, {\sc ---}, {\sc ---},
Selfadjoint elliptic operators and
manifold decompositions Part II: Spectral flow and Maslov index,
{\it Comm. Pure Appl. Math. {\bf 49}} (1996), 869--909.


\bibitem{ChMa74} \textsc{P.R. Chernoff and J.E. Marsden}, {\it Properties of
Infinite Dimensional Hamiltonian Systems}, LNM {\bf 425},
Springer-Verlag, Berlin, 1974.

\bibitem{Cr98} \textsc{R. Cross}, {\em Multivalued Linear Operators}, Dekker Inc., New York,
1998.

\bibitem{DaZh98} {\sc X. Dai and W. Zhang}, Higher spectral flow,
{\it J. Funct. Anal.} {\bf 157} (1998), 432--469.


\bibitem{Du76} \textsc{J.J. Duistermaat}, On the Morse index in variational
calculus, {\it Adv. Math.} {\bf 21} (1976), 173--195.

\bibitem{Fl88} {\sc A. Floer}, A relative Morse index for the symplectic
action, {\it Comm. Pure Appl. Math.} {\bf 41} (1988), 393--407.


\bibitem{FuOt02} {\sc K. Furutani and N. Otsuki}, Maslov index in the infinite dimension
and a splitting formula for a spectral flow, {\it Japan. J. Math.}
{\bf 28}/2 (2002),  215--243.

\bibitem{Go01}{\sc M. de Gosson},
\textit{The Principles of Newtonian and Quantum Mechanics - With a Forword by Basil Hiley},
Imperial College / World Scientific Publishing Co., London--Singapore, 2001.





\bibitem{Ka76}\textsc{T. Kato}, \textit{Perturbation Theory for Linear
Operators}, Springer-Verlag, Berlin, 1966, 2d ed., 1976.

\bibitem{KiLe00} {\sc P. Kirk and M. Lesch},
 {The $\eta$--invariant, Maslov index, and spectral flow
for Dirac--type operators on manifolds with boundary}, {\it Forum Math.} {\bf 16}
(2004), 553--629, {arXiv:math.DG/0012123}.



\bibitem{Le78} {\sc J. Leray}, \textit{Analyse Lagrangi\'enne et
m\'ecanique quantique: Une structure math\'ematique apparent\'ee aux
d\'eveloppements asymptotiques et \`a l'indice de Maslov}, S\'erie
Math. Pure et Appl., I.R.M.P., Strasbourg, 1978 (English translation
1981, MIT Press).

\bibitem{LeMal03} {\sc M. Lesch and M. Malamud}, On the deficiency indices and self--adjointness
of symmetric Hamiltonian systems, {\it J. Differential Equations}
{\bf 189} (2003), no. 2, 556--615.






\bibitem{Mo} \textsc{M. Morse}, \textit{The Calculus of Variations in the
  Large}, A.M.S. Coll. Publ., Vol.18, Amer. Math. Soc., New York, 1934.

\bibitem{MuPePo} {\sc M. Musso, J. Pejsachowicz, and A. Portaluri},
A Morse index theorem for perturbed geodesics on semi-Riemannian manifolds, {\it Topol. Methods
Nonlinear Anal.} {\bf 25}/1 (2005), 69--99.

\bibitem{Ni95}
{\sc L. Nicolaescu}, The Maslov index, the spectral flow, and
decomposition of manifolds, {\it Duke Math. J.} {\bf 80} (1995),
485--533.



\bibitem{Ph96}
{\sc J. Phillips}, Self--adjoint Fredholm operators and spectral
flow, {\it Canad. Math. Bull.} {\bf 39} (1996), 460--467.

\bibitem{PiT1} \textsc{P. Piccione and D.V. Tausk}, The Maslov index and a
generalized Morse index theorem for non-positive definite metrics,
{\it C. R. Acad. Sci. Paris S{\'e}r. I Math.} {\bf 331} (2000),
385--389.

\bibitem{PiT2} {\sc ---, ---}, The Morse index theorem
in semi-Riemannian Geometry, {\it Topology} {\bf 41} (2002), 1123--1159, arXiv:math.DG/0011090.

\bibitem{Pro:SFD} \textsc{M. Prokhorova}, The spectral flow for
Dirac operators on compact planar domains with local boundary
conditions, 33 pages, arXiv:1108.0806v3 [math-ph].

\bibitem{RoSa93} {\sc J. Robbin and D. Salamon},
The Maslov index for paths, {\it Topology {\bf 32}} (1993),
827--844.



\bibitem{Sch01} \textsc{B.-W. Schulze}, An algebra of boundary value problems not requiring
Shapiro--Lopatinskij conditions, \textit{J. Funct. Anal.} {\bf 179} (2001), 374--408.







\bibitem{Wa} {\sc N. Waterstraat}, A $K$-theoretic proof
of the Morse index theorem in semi-Riemannian geometry, {\it Proc. Amer. Math. Soc.} {\bf 140}/1
(2012), 337--349.

\bibitem{Wo85} {\sc K.P. Wojciechowski}, Spectral flow and the general
linear conjugation problem, {\it Simon Stevin} {\bf 59} (1985), 59--91.

\bibitem{Yo91}
{\sc T. Yoshida}, Floer homology and splittings of manifolds,
{\it Ann. of Math.} {\bf 134} (1991),  277--323.

\bibitem{Zh00} {\sc C. Zhu}, \textit{Maslov-type index theory and closed
characteristics on compact convex hypersurfaces in $\R^{2n}$}, PhD
Thesis (in Chinese), Nankai Institute, Tianjin, 2000.

\bibitem{Zh01} {\sc ---}, \textit{The Morse Index Theorem for Regular
Lagrangian Systems}, Preprint September 2001 ({math.DG/0109117})
(first version). MPI Preprint. 2003. No. 55 (modified version).

\bibitem{Zh05} {\sc ---}, {A generalized Morse index theorem},  {\it in}:
B. Boo{\ss}--Bavnbek et al. (eds.),
``Analysis, Geometry and Topology of Elliptic Operators", World Scientific,
London and Singapore, 2006, pp. 493--540.


\bibitem{ZhLo99} {\sc C. Zhu and Y. Long}, Maslov-type index theory for
symplectic paths and spectral flow. (I), \textit{Chinese Ann. of
Math.} {\bf 20}B (1999), 413--424.

\end{thebibliography}
\end{document}